\documentclass[11pt,reqno]{amsart}
\usepackage[left=33mm,right=33mm,top=30mm,bottom=32mm]{geometry}
\usepackage{mathtools,amssymb,amsthm,mathrsfs,color,float}
\usepackage{paralist}
\usepackage{stackengine}
\usepackage{centernot}
\usepackage{txfonts}
\usepackage{tabularx}
\usepackage{longtable}
\usepackage[colorlinks,
linkcolor=red,
anchorcolor=green,
citecolor=blue, 
]{hyperref}

\usepackage[T1]{fontenc}
\usepackage[utf8]{inputenc}
\usepackage{booktabs}

\usepackage{calc}
\definecolor{bleu1}{RGB}{0,57,128}
\def\bleu1{\color{bleu1}}

\usepackage{etoolbox}
\patchcmd{\section}{\normalfont}{\normalfont \bleu1}{}{}
\patchcmd{\subsection}{\normalfont}{\normalfont \bleu1}{}{}
\patchcmd{\subsubsection}{\normalfont}{\normalfont \bleu1}{}{}

\newtheorem{The}{\bleu1 Theorem}[section]
\newtheorem{Pro}{\bleu1 Proposition}[section]
\newtheorem{Lem}{\bleu1 Lemma}[section]
\newtheorem{Cor}{\bleu1 Corollary}[section]

\theoremstyle{definition}
\newtheorem{defn}{\bleu1 Definition}[section]

\newtheorem{Rem}{\bleu1 Remark}[section]
\newtheorem{Ex}{\bleu1 Example}[section]

\newcommand{\T}{\mathbb{T}}
\newcommand{\R}{\mathbb{R}}

\newcommand{\N}{\mathbb{N}}

\newcommand{\SCL}{\mathrm{SCL\,}}

\newcommand{\reg}{\operatorname{Reg}}
\newcommand{\cut}{\operatorname{Cut}}
\newcommand{\sing}{\operatorname{Sing}}

\makeatletter
\@namedef{subjclassname@2020}{\textup{2020} Mathematics Subject Classification}
\makeatother

\title[On the cut locus of Hamilton--Jacobi equations I]{On the cut locus of Hamilton--Jacobi equations I: structure and propagation via the touching approach}
\author{Piermarco Cannarsa, Wei Cheng, Jiahui Hong \and Wenxue Wei}
\address[Piermarco Cannarsa]{Dipartimento di Matematica, Universit\`a di Roma ``Tor Vergata'', Via della Ricerca Scientifica 1, 00133 Roma, Italy}
\email{cannarsa@mat.uniroma2.it}
\address[Wei Cheng]{School of Mathematics, Nanjing University, Nanjing 210093, China}
\email{chengwei@nju.edu.cn}
\address[Jiahui Hong]{School of Mathematics, Nanjing University of Aeronautics and Astronautics, Nanjing 211106, China}
\email{hongjiahui@nuaa.edu.cn}
\address[Wenxue Wei]{School of Mathematics, Nanjing University, Nanjing 210093, China}
\email{wwx3708@gmail.com}

\begin{document}

\begin{abstract}
For a semiconcave function with linear modulus, we introduce the cut locus through a touching approach and prove that it coincides with the variational cut locus defined via the Lax--Oleinik semigroup, independently of the Hamiltonian. Cut points are characterized by the emptiness of the proximal subdifferential of $\phi$, which we use as the analytic criterion behind the touching definition. We introduce the degree of regularity $R_\phi$, which measures the $C^{1,1}$ deviation of $\phi$, and prove the quantitative relation $R_\phi(x)\sim 1/\tau_{\phi,H}(x)$ with the cut time function, valid pointwise up to explicit truncation constants. From this relation we derive a separation estimate for calibrated curves, showing that no conjugate points occur before the cut locus. This estimate supplies the tools for the propagation results. Cut points propagate globally along generalized characteristics for the evolutionary Hamilton--Jacobi equation, and Alexandrov points propagate forward along calibrated curves, with the second derivative satisfying a matrix Riccati equation. We also show that the cut locus is a Lebesgue null set and give a streamlined proof of Alexandrov's theorem. Finally, we give necessary and sufficient conditions for the cut locus of a weak KAM solution to be closed, in terms of the cut time function, the $C^{1,1}$ support, and the degree of regularity.  
\end{abstract}

\maketitle


\section{Introduction}

This is the first of two related papers on the cut locus of viscosity solutions of Hamilton--Jacobi equations. The second paper will concern the quantitative regularity and stability of the cut time function and of the cut locus under $C^{2,\alpha}$ initial data.

\subsection{Hamilton--Jacobi equations and the cut locus}

Let $M$ be a compact connected smooth manifold without boundary, and let $TM$ and $T^*M$ denote its tangent and cotangent bundles.  A Hamiltonian $H: T^*M \to \R$ is called \emph{Tonelli} if the following two conditions hold.
\begin{enumerate}[\rm (H1)]
\item $H$ is of class $C^2$ on $T^*M$, and there exists a nondecreasing function $g_H:[0,+\infty)\to\R$ with $g_H(r)/r\to+\infty$ as $r\to+\infty$ such that
\begin{equation}\tag{H1}
H(x,p)\geqslant g_H(|p|),\qquad \forall\,(x,p)\in T^*M,
\end{equation}
\item there exists a nonincreasing function $\alpha_H:[0,+\infty)\to(0,+\infty)$ such that
\begin{equation}\tag{H2}
H_{pp}(x,p)\geqslant \alpha_H(|p|)\, I,\qquad \forall\,(x,p)\in T^*M,
\end{equation}
\end{enumerate}
where $|\cdot|$ denotes the norm on $T^*M$ induced by a fixed Riemannian metric on $M$ and $I$ is the identity matrix.  Note that $g_H$ is allowed to take negative values. This is necessary because the normalization $H\mapsto H-c[H]$ does not in general preserve a lower bound of the form (H1) with $g_H\geqslant0$ (for instance, a normalized mechanical Hamiltonian may be negative at some $(x,0)$).  We call $g_H$ and $\alpha_H$ the \emph{superlinear growth modulus} and the \emph{local uniform convexity modulus} of $H$, respectively, and refer to the pair $(g_H,\alpha_H)$ as the \emph{Tonelli modulus} of $H$.  The associated \emph{Tonelli Lagrangian} $L: TM \to \R$ is defined by
\[
L(x,v) = \sup_{p \in T^*_x M} \bigl\{ \langle p, v \rangle - H(x,p) \bigr\},
\qquad x \in M,\; v \in T_x M.
\]
We consider the evolutionary Hamilton--Jacobi equation
\begin{equation}\label{eq:HJe}\tag{HJ$_e$}
D_t u(t,x) + H(x, D_x u(t,x)) = 0, \qquad t > 0,\; x \in M,
\end{equation}
and the stationary Hamilton--Jacobi equation
\begin{equation}\label{eq:HJs}\tag{HJ$_s$}
H(x, D\phi(x)) = 0, \qquad x \in M,
\end{equation}
where the Ma\~n\'e critical value $c[H]$ of $H$ has been normalized to $0$, which is always possible by replacing $H$ with $H - c[H]$.  This normalization is kept throughout the paper.

Viscosity solutions of \eqref{eq:HJe} and \eqref{eq:HJs} are, in general, no better than locally semiconcave.  Semiconcave functions were employed to study well-posedness for \eqref{eq:HJe} even before the theory of viscosity solutions was developed \cite{Douglis1961, Kruzkov1975, Krylov_book1987}.  The notion of viscosity solution, introduced in the seminal papers \cite{Crandall_Lions1983, Crandall_Evans_Lions1984}, supplies the appropriate generalized-solution framework for the well-posedness questions of existence, uniqueness, and stability for \eqref{eq:HJe}.  An overview of the main features of this theory can be found in the monographs \cite{Bardi_Capuzzo-Dolcetta1997, Fleming_Soner_book2006}.  Nowadays, semiconcave functions find wide application across diverse mathematical fields, such as control theory and sensitivity analysis \cite{Hrustalev1978, Cannarsa_Frankowska1991, Fleming_McEneaney2000, Rifford2000, Rifford2002}, nonsmooth and variational analysis \cite{Rockafellar1982, Colombo_Marigonda2006}, and metric geometry \cite{Petrunin2007}.  Comprehensive treatments of semiconcave functions can be found in the monographs \cite{Cannarsa_Sinestrari_book, Villani_book2009, Ambrosio_GigliNicola_Savare_book2008}.

Beyond semiconcavity, finer regularity properties have been obtained through Lasry--Lions-type regularization methods.  Bernard \cite{Bernard2007} established the existence of optimal $C^{1,1}$ critical subsolutions of \eqref{eq:HJs}.  Arnaud \cite{Arnaud2011} gave a geometric and dynamical interpretation of the Lax--Oleinik semigroup via pseudographs, connecting the regularization viewpoint with weak KAM theory.  Cannarsa, Cheng, and Hong \cite{Cannarsa_Cheng_Hong2025} developed a unified regularization framework for weak KAM solutions and their cut loci, which provides the technical foundation for the touching approach used in the present paper.  At the level of the gradient, Bianchini and Tonon \cite{Bianchini_Tonon2012} proved that $D_x u$ belongs to SBV except for at most countably many times, which sharply constrains the Cantor part of the singularity measure.  Rifford \cite{Rifford2008} obtained further regularity results and generalized Sard theorems as well.  In the same touching spirit as the present paper, the Mather quotient problem was addressed geometrically in \cite{Cheng_Wei2026}.

Given a viscosity solution $u$ of \eqref{eq:HJe} with initial datum $u(0,\cdot) = \phi$, the Lax--Oleinik representation $u(t,x) = T_t^- \phi(x)$ holds \cite{Bernard2012}, and $u$ is locally semiconcave with linear modulus. A curve $\gamma: I \to M$ is called \emph{$(u,H)$-calibrated} if
\[
u(t',\gamma(t')) - u(t,\gamma(t)) = \int_t^{t'} L(\gamma(s),\dot\gamma(s))\, ds,
\qquad \forall\,[t,t'] \subset I.
\]
A point $(t,x)$ is \emph{regular} for $(u,H)$ if there exists $\delta>0$ and a calibrated curve $\gamma:[t,t+\delta]\to M$ with $\gamma(t)=x$. Otherwise it is a \emph{cut point}.  The \emph{cut locus} $\cut(u,H)$ is the set of all cut points. The \emph{cut time function} for $(\phi,H)$ is
\begin{equation}
\tau_{\phi,H}(x) = \sup\bigl\{ t \geqslant 0 : \exists\ \text{a $(u,H)$-calibrated curve}\ \gamma \in C^1([0,t],M)\ \text{with}\ \gamma(0) = x \bigr\},
\end{equation}
and the cut locus for $(\phi,H)$ is defined by $\cut(\phi,H)=\{x\in M:\tau_{\phi,H}(x)=0\}$.  We also denote by $\reg(\phi,H)$ the set of regular points for $(\phi,H)$, namely $\reg(\phi,H)=\{x\in M:\tau_{\phi,H}(x)>0\}$, so that $\cut(\phi,H)=M\setminus\reg(\phi,H)$.

\subsection{Cut locus via the touching approach}

The notion of cut locus originates in Riemann--Finsler geometry \cite{Klingenberg_book1982, Sakai_book1996, Berger_book2003, Itoh_Tanaka2001, Li_Nirenberg2005, Miura_Tanaka2024, Figalli_Rifford_Villani2011}, but it extends naturally to any semiconcave function.  Let $\phi$ be a semiconcave function with linear modulus on an open convex set $\Omega \subset \R^n$.  We call $x \in \Omega$ a \emph{regular point} of $\phi$ if there exist a convex neighborhood $V$ of $x$ and a semiconvex function $\psi: V \to \R$ with linear modulus touching $\phi$ from below at $x$.  We denote by $\reg(\phi)$ the set of regular points of $\phi$.  The \emph{cut locus} of $\phi$ is $\cut(\phi) = \Omega \setminus \reg(\phi)$.  This definition is equivalent to the non-emptiness of the proximal subdifferential, a standard notion from nonsmooth analysis (Proposition \ref{pro:prox}), which we adopt as the main criterion for cut points. In the time-dependent setting, Albano, Cannarsa, and Sinestrari \cite{Albano_Cannarsa_Sinestrari2020} used a similar proximal criterion to detect the generation of singularities from the initial datum of a Hamilton--Jacobi equation. This definition uses only the local semiconcavity structure, and for any Tonelli Hamiltonian $H$ it is equivalent to the variational one, in the sense that
\begin{equation}\label{eq:intro equi}
\reg(\phi) = \reg(\phi,H), \qquad \cut(\phi) = \cut(\phi,H).
\end{equation}
Thus the cut locus is an intrinsic analytic object associated to the semiconcave function itself, independent of the particular Hamiltonian (Theorem~\ref{thm:Reg_Cut}).

Using the equivalence \eqref{eq:intro equi} and a variational argument, we prove that $\cut(\phi)$ is a Lebesgue null set (Theorem~\ref{thm:cut0}).  As a corollary we obtain a streamlined proof of Alexandrov's theorem on the almost everywhere twice differentiability of semiconcave functions (Theorem~\ref{thm:Alex}).

We introduce the \emph{degree of regularity} $R_\phi(x)$, which measures the local $C^{1,1}$ deviation of $\phi$ near a regular point (Definition~\ref{def:R_phi}). Theorem~\ref{thm:1} establishes the pointwise equivalence $R_\phi(x)\sim 1/\tau_{\phi,H}(x)$ with explicit truncated one-sided bounds, effective for small cut times, that is, near the cut locus. This relation links the touching and the variational descriptions quantitatively. The function $R_\phi$ is the analytic object of the touching approach, and $\tau_{\phi,H}$ measures the maximal length of calibrated curves. As a first application, the estimate $R_\phi(x)\sim 1/\tau_{\phi,H}(x)$ yields necessary and sufficient conditions for $\cut(\phi)$ to be closed when $\phi$ is a weak KAM solution (Theorem~\ref{thm:close_cut}). When closed, the cut locus is a geometric object in the usual sense. The separation estimates for calibrated curves (Theorem~\ref{distance estimate}) rest on the same relation and supply the tools for the propagation results. Geometrically, they show that, from the trajectory point of view, no conjugate points occur before the cut locus along a calibrated curve, for arbitrary semiconcave initial data.

\subsection{Propagation of cut points and Alexandrov points}

The persistence of singularities (meaning that once a singularity appears, it propagates forward in time up to $+\infty$) provides one indication of the irreversibility of the evolutionary equation \eqref{eq:HJe} (see for instance \cite{Albano_Cannarsa2002, Cannarsa_Yu2009} for the local propagation, and \cite{ACNS2013, Cannarsa_Mazzola_Sinestrari2015, Cannarsa_Cheng_Hong2025, Albano_Cannarsa_Cheng_Mendico2026} for the global propagation), while the compactness of the evolution under the associated Lax--Oleinik semigroup gives another \cite{Ancona_Cannarsa_Nguyen2016_1, Ancona_Cannarsa_Nguyen2016_2}.

In \cite{Albano_Cannarsa2002}, the authors introduced the concept of \emph{generalized characteristics} for the Hamilton--Jacobi equation \eqref{eq:HJe}, a notion that has since become standard.  The structure of singularities of semiconcave solutions was further investigated by Rifford \cite{Rifford2002, Rifford2003}, who gave a detailed classification of singular points.  In the one-dimensional case, the idea of generalized characteristics can also be traced back to earlier work by Dafermos \cite{Dafermos1977} on the Burgers equation.  For surveys of the singularities of viscosity solutions of Hamilton--Jacobi equations, we refer the reader to \cite{Cannarsa_Cheng2021a, Cheng_Hong2026}.  

A Lipschitz curve $\gamma: [t_0,\infty) \to M$ with $t_0>0$ is called a \emph{generalized characteristic} for $(u,H)$ if it satisfies the differential inclusion
\begin{equation}\label{eq:intro_gc}
\dot{\gamma}(s) \in \operatorname{co} H_p\bigl(\gamma(s), D^+_x u(s, \gamma(s))\bigr), \qquad \text{a.e. } s \geqslant t_0.
\end{equation}
It was proved in \cite{Albano_Cannarsa2002} that a generalized characteristic exists globally from any initial point $x_0$, and propagates the singularity locally if $x_0$ is singular and $0 \notin \operatorname{co} H_p(x_0, D^+ u(x_0))$ (see \cite{Yu2006} for improvements).  Inspired by earlier works \cite{Bogaevsky2002, Bogaevski2006, Stromberg2013}, Khanin and Sobolevski essentially established the existence of singular characteristics satisfying \eqref{eq:intro_gc} without the convex hull under additional conditions on the initial data \cite{Khanin_Sobolevski2016}.  Such singular characteristics are called \emph{strict singular characteristics} (or \emph{broken characteristics} in \cite{Stromberg2013}).

In this paper, we prove two propagation results.  First, we show that any generalized characteristic starting from a cut point remains in the cut locus for all future times (Theorem~\ref{thm:pro_cut}), strengthening the classical propagation theorems for the closure of the singular set (for the stationary case, see \cite[Theorem 5.7]{CCHW2024}). Second, we prove that Alexandrov points (points of twice differentiability) propagate forward along calibrated curves and that the second derivative evolves according to a matrix Riccati equation (Theorem~\ref{the:propagate alex}).  The classical theory of Alexandrov point propagation requires the initial datum to be of class $C^2$. Here the result is established for semiconcave initial data, which is the general regularity for the Cauchy problem~\eqref{eq:HJe}. This extension is useful for two reasons. It supplies the second-order analysis needed to study strict singular characteristics in the general case, where the initial datum is only semiconcave. It also makes the second derivative $Q(t) = D^2_xu(t,\gamma(t))$ along calibrated curves well defined away from the cut locus, without restricting to a $C^2$ neighborhood of the Mather set.  The result also clarifies a subtle point in the propagation theory. Alexandrov points always propagate forward along calibrated curves, while backward propagation requires the additional condition that the point lies in the $C^{1,1}$-support of the solution (see Theorem~\ref{the:propagate alex}~(2)).

\medskip

Of the results presented in this paper, the following are new:
\begin{enumerate}[\rm (1)]
\item the equivalence, for a general semiconcave function, between the cut locus defined by touching and the variational cut locus $\cut(\phi) = \cut(\phi,H)$ (Theorem~\ref{thm:Reg_Cut}). This equivalence is new, though related characterizations exist in the literature \cite{Cannarsa_Sinestrari_book}. Here cut points are characterized by the emptiness of the proximal subdifferential $\partial_P\phi$, a criterion also used, in the time-dependent setting, in \cite{Albano_Cannarsa_Sinestrari2020}.
\item the notion of the degree of regularity $R_\phi$ and its quantitative relation $R_\phi(x)\sim 1/\tau_{\phi,H}(x)$ to the cut time function $\tau_{\phi,H}$ (Theorem~\ref{thm:1}).
\item the separation estimate for calibrated curves showing that no conjugate points occur before the cut locus, for arbitrary semiconcave initial data (Theorem~\ref{distance estimate}).
\item the propagation of cut points along generalized characteristics for the evolutionary equation \eqref{eq:HJe} (Theorem~\ref{thm:pro_cut}), which extends \cite[Theorem 5.7]{CCHW2024} from the stationary to the evolutionary setting, and the propagation of Alexandrov points with the matrix Riccati equation (Theorem~\ref{the:propagate alex}).
\end{enumerate}
Earlier propagation results were restricted to singular (non-differentiability) points and did not track the evolution of second derivatives.  The Lebesgue null property of $\cut(\phi)$, the streamlined proof of Alexandrov's theorem, and the closedness criterion for weak KAM solutions (cf. \cite{Cannarsa_Cheng_Fathi2021}) recover known facts within the unified framework of the touching approach.

\medskip

This paper is organized as follows. Section~2 collects preliminaries on semiconcave functions and Hamilton--Jacobi equations, including the definition of the cut time function and some fundamental results on it (Definition~\ref{defn:tau_phi} and Proposition~\ref{pro:tau-usc}). Section~3 develops the touching approach to the cut locus. It gives the definition of the cut locus of a semiconcave function and its equivalent characterizations (Definition~\ref{defn:various_sets}, Proposition \ref{pro:prox} and Proposition~\ref{pro:touching_diff_reg}), introduces the degree of regularity $R_\phi$ and its quantitative relation to the cut time function $\tau_{\phi,H}$ (Theorem~\ref{thm:1}), and proves the equivalence between the touching and the variational definitions (Theorem~\ref{thm:Reg_Cut}). The section also treats the Lebesgue null property of the cut locus (Theorem~\ref{thm:cut0}), a streamlined proof of Alexandrov's theorem (Theorem~\ref{thm:Alex}), the closedness criterion for the cut locus of weak KAM solutions (Theorem~\ref{thm:close_cut}), the separation estimate for calibrated curves (Theorem~\ref{distance estimate}), and the propagation of cut points and Alexandrov points (Theorems~\ref{thm:pro_cut} and~\ref{the:propagate alex}).

\medskip

The second paper will also study the quantitative stability of the cut time function and of the cut locus under $C^{2,\alpha}$ perturbations. Our Theorem \ref{thm:1} on the quantitative relation $R_\phi(x)\sim 1/\tau_{\phi,H}(x)$ in this article will be a key tool to obtain those results.

\medskip

\begin{sloppypar}
\noindent\textbf{Acknowledgements.} Piermarco Cannarsa was supported, in part, by the National Group for Mathematical Analysis, Probability and Applications of the Italian Istituto Nazionale di Alta Matematica ``Francesco Severi'', and acknowledges support from the Excellence Department Project awarded to the Department of Mathematics, University of Rome Tor Vergata, CUP E83C23000330006, and the European Union--Next Generation EU, PRIN 2022 PNRR (CUP E53D23017910001).  Wei Cheng is partly supported by National Natural Science Foundation of China (Grant No.~12231010).  Jiahui Hong is partly supported by National Natural Science Foundation of China (Grant No.~12501245).
\end{sloppypar}

\section{Preliminaries}

\subsection{Semiconcave function}

Let \( \Omega \subset \R^n \) be an open convex subset. A function $\phi:\Omega\to\R$ is called \emph{semiconcave} with linear modulus $C\geqslant0$ if 
\begin{align*}
	\lambda\phi(x)+(1-\lambda)\phi(y)-\phi(\lambda x+(1-\lambda)y)\leqslant \frac C2\lambda(1-\lambda)|x-y|^2,\quad \forall x,y\in \Omega,\ \lambda\in[0,1].
\end{align*}
We also say $\phi$ is $C$-semiconcave on $\Omega$ and call $C$ the semiconcavity constant of $\phi$. We denote by $\text{\rm SCL}\,(\Omega)$ the family of all semiconcave functions with linear modulus on $\Omega$. It is well known that $\phi\in\text{\rm SCL}\,(\Omega)$ is locally Lipschitz continuous. A function $\phi:\Omega\to\R$ is called \emph{semiconvex} if $-\phi$ is semiconcave.

For real-valued functions $f,g$ on $\Omega$ we say \emph{$g$ touches $f$ from below at $x$ in $\Omega$} if 
\begin{align*}
	g(x)=f(x)\quad\text{and}\quad g(y)\leqslant f(y),\qquad\forall y\in \Omega.
\end{align*}
Similarly, we say $g$ touches $f$ from above at $x$ if $-g$ touches $-f$ from below at $x$.

Let $\phi$ be a function on $\Omega$. The \emph{superdifferential} and \emph{subdifferential} of $\phi$ are defined as
\begin{align*}
	D^{+}\phi(x)&=\{D\psi(x):\exists\ r>0,\ \psi\in C^1(B(x,r)),\ \text{s.t. $\psi$ touches $\phi$ from above at $x$ in $B(x,r)$}\},\ x\in\Omega,\\
	D^{-}\phi(x)&=\{D\psi(x):\exists\ r>0,\ \psi\in C^1(B(x,r)),\ \text{s.t. $\psi$ touches $\phi$ from below at $x$ in $B(x,r)$}\},\ x\in\Omega,
\end{align*}
where $B(x,r)$ denotes the open ball centered at $x$ with radius $r$.

For $\phi\in\text{\rm SCL}\,(\Omega)$, it is well known that $D^{+}\phi(x)$ is a nonempty, convex, compact subset of $\R^n$ for any $x\in\Omega$. $D^{+}\phi(x)$ is a singleton if and only if $\phi$ is differentiable at $x$, and in this case we have $D^{+}\phi(x)=D\phi(x)$. The set-valued map $x\rightrightarrows D^+\phi$ is upper-semicontinuous on $\Omega$. Similar properties hold for semiconvex functions and their subdifferential. The readers can refer to \cite[\S3]{Cannarsa_Sinestrari_book} for more details on semiconcave functions.

\begin{Pro}[\protect{\cite[Proposition 3.3.1]{Cannarsa_Sinestrari_book}}]\label{pro:scl fund}
	Let \( \Omega \subset \R^n \) be an open convex subset and $\phi$ be a semiconcave function with linear modulus $C>0$ on $\Omega$. Then for any $x\in\Omega$ and $p\in D^+\phi(x)$, the function $\psi(y)=\phi(x)+p\cdot(y-x)+\frac{C}{2}|y-x|^2$, $y\in\Omega$ touches $\phi$ from above at $x$ in $\Omega$.
\end{Pro}

\begin{Pro}[\protect{\cite[Proposition 2.5]{Cannarsa_Cheng_Hong2025}}]\label{pro:touching}
Let \( \Omega \subset \R^n \) be an open convex subset, $\phi\in\text{\rm SCL}\,(\Omega)$ and $\psi:\Omega\to\R$ be semiconvex with linear modulus. If $\psi$ touches $\phi$ from below at $x$, then both $\psi$ and $\phi$ are differentiable at $x$ and we have $D\psi(x)=D\phi(x)$.
\end{Pro}

\begin{defn}
	Let \( \Omega \subset \R^n \) be an open convex subset and $\phi\in\text{\rm SCL}\,(\Omega)$.
	\begin{enumerate}[(1)]
		\item We define the \emph{differentiability set} of \( \phi \) as 
		\[
		\text{Diff}\,(\phi)= \{ x \in \Omega : \phi \text{ is differentiable at } x \},
		\]
		and the \emph{singular set} of $\phi$ as $\sing(\phi)=\Omega\setminus\text{Diff}\,(\phi)$.
		\item We denote by \( \text{Alex}\,(\phi) \) the set of points in \( \Omega \) where \( \phi \) is twice differentiable, i.e.,
		\begin{align*}
			\text{Alex}\,(\phi)= \big\{ x \in \Omega :\ & \text{there exists }D\phi(x)\text{ and }D^2\phi(x) \text{ such that} \\
			& \phi(y) = \phi(x) + \langle D\phi(x), y-x \rangle + \frac{1}{2} \langle D^2\phi(x)(y-x), y-x \rangle + o(|y-x|^2) \big\}.
		\end{align*}
	\end{enumerate}
\end{defn}

\begin{Pro}[\protect{\cite[Theorem 1.2]{Azagra2023}}]\label{pro:Alex2}
	Let $\Omega \subset \R^n$ be an open and convex set, and let $\phi \in \operatorname{SCL}(\Omega)$.
	Then for any $x \in \operatorname{Diff}(\phi)$ and any symmetric real $n \times n$ matrix $A$,
	we have
	\begin{align*}
		\lim_{y \to x} \frac{\bigl|\phi(y) - \phi(x) - \langle D\phi(x), y-x \rangle
			- \tfrac{1}{2} \langle A(y-x), y-x \rangle\bigr|}{|y-x|^2} = 0
	\end{align*}
	if and only if
	\begin{align*}
		\lim_{y \to x} \sup_{p \in D^{+}\phi(y)} \frac{|p - D\phi(x) - A(y-x)|}{|y-x|} = 0.
	\end{align*}
\end{Pro}

Let $M$ be a smooth manifold without boundary. We say $\phi:M\to\R$ is \emph{locally semiconcave} with linear modulus if for any $x\in M$ there exists an open neighborhood $U$ of $x$ with local coordinate such that $\phi\vert_U$ is semiconcave with linear modulus. We denote by $\text{\rm SCL}_{\rm loc}\,(M)$ the space of locally semiconcave functions with linear modulus on $M$. If $M$ is compact, we just denote $\text{\rm SCL}\,(M)$ instead of $\text{\rm SCL}_{\rm loc}\,(M)$, since $M$ can be covered by finitely many coordinate charts. A function $\phi:M\to\R$ is called \emph{locally semiconvex} if $-\phi$ is locally semiconcave.

\subsection{Hamilton--Jacobi equation}

\begin{sloppypar}
Let $M$ be a compact connected smooth manifold without boundary, let $L:TM\to\R$ be a Tonelli Lagrangian with associated Hamiltonian $H$. For any function $\phi\in C^0(M)$ we define the Lax--Oleinik semigroups $\{T^{\pm}_t\}_{t\geqslant0}$ by
\end{sloppypar}
\begin{align*}
	T^+_t\phi(x)=\sup_{y\in M}\{\phi(y)-A_t(x,y)\},\quad T^-_t\phi(x)=\inf_{y\in M}\{\phi(y)+A_t(y,x)\},\quad x\in M,
\end{align*}
where $A_t(x,y)$ is the fundamental solution given by
\begin{align*}
	A_t(x,y)=\inf_{\xi}\int^t_0L(\xi(s),\dot{\xi}(s))\ ds,\quad t>0,\ x,y\in M,
\end{align*}
with the infimum taken over the set of absolutely continuous curves $\xi:[0,t]\to M$ connecting $\xi(0)=x$ to $\xi(t)=y$. For convenience, we set $T^{\pm}_0 = \mathrm{id}$. It is well known that both $A_t(x,\cdot)$ and $A_t(\cdot,x)$ belong to $\text{\rm SCL}(M)$ for any $t>0$ and $x\in M$, and the semigroups $\{T^{\pm}_t\}_{t\geqslant0}$ are order-preserving.

The following quantitative regularity result for the fundamental solution will be used in Section~3.

\begin{Pro}[\protect{\cite[Propositions B.3 and B.8]{Cannarsa_Cheng3}}]\label{pro:At}
	Let $H:\T^n\times\R^n\to\R$ be a Tonelli Hamiltonian and let $A_t(x,y)$ be the associated fundamental solution. Then the following statements hold.
	\begin{enumerate}[\rm (1)]
		\item For every $R>0$ there exist $t_R>0$ and $K_R>0$, depending on $R$ and on the local $C^2$ modulus and the Tonelli modulus of $H$, such that, for every $0<t\leqslant t_R$ and $y\in\T^n$, the function $x\mapsto A_t(x,y)$ is semiconcave with constant $K_R/t$ on $B(y,Rt)$, understood in local flat coordinates. The threshold $t_R$ is chosen small enough that these balls lie within the injectivity radius.
		\item For every $R>0$ there exist $\tau_0(R)\in(0,1]$ and $K(R)>0$, depending on $R$ and on the local $C^2$ modulus and the Tonelli modulus of $H$, such that for every $t\in(0,\tau_0(R)]$ and $y\in\T^n$ the function $x\mapsto A_t(x,y)$ is of class $C^2$ on the ball $B(y,Rt)$ and satisfies
		\begin{equation}
			D_{xx}A_t(x,y)\geqslant\frac{K(R)}{t}\,I,\qquad\forall x\in B(y,Rt).
		\end{equation}
	\end{enumerate}
	The result for the initial endpoint follows from the corresponding result for the terminal endpoint by using the reversed Lagrangian $\widehat L(x,v)=L(x,-v)$, whose action satisfies $\widehat A_t(y,x)=A_t(x,y)$.
\end{Pro}


Recall that a continuous function $u:(0,\infty)\times M\to\R$ is called a \textit{viscosity solution} of the evolutionary Hamilton--Jacobi equation \eqref{eq:HJe} if the following two conditions hold for every $(t,x)\in(0,\infty)\times M$:
\begin{enumerate}[--]
	\item For any $C^1$ function $v:(0,\infty)\times M\to\R$ touching $u$ from above at $(t,x)$, we have
	\begin{align*}
		D_tv(t,x)+H(x,D_xv(t,x))\leqslant 0.
	\end{align*}
	\item For any $C^1$ function $v:(0,\infty)\times M\to\R$ touching $u$ from below at $(t,x)$, we have
	\begin{align*}
		D_tv(t,x)+H(x,D_xv(t,x))\geqslant 0.
	\end{align*}
\end{enumerate}
It is well known that $u\in\text{\rm SCL}_{\rm loc}\,((0,\infty)\times M)$ and $u(t,\cdot)\in\text{\rm SCL}\,(M)$ for all $t>0$. The unique viscosity solution of \eqref{eq:HJe} with initial data $u(0,\cdot)=\phi$, where $\phi\in C^0(M)$, has the representation formula $u(t,x)=T^-_t\phi(x)$, $(t,x)\in[0,\infty)\times M$.

A function $\phi:M\to\R$ is called a \emph{weak KAM solution} for the stationary Hamilton--Jacobi equation \eqref{eq:HJs} if $T^-_t\phi=\phi$ for all $t\geqslant0$. For a Tonelli Hamiltonian $H$, the notions of weak KAM solution and viscosity solution of \eqref{eq:HJs} coincide. The readers can find more details on weak KAM theory under our conditions in \cite{Fathi_Siconolfi2005}.

\subsection{Cut time function} 
Let $M$ be a compact connected smooth manifold without boundary, $H:T^*M\to\R$ be a Tonelli Hamiltonian and $\phi\in \text{\rm SCL}(M)$. Let $u$ be the viscosity solution to the Cauchy problem of evolutionary Hamilton--Jacobi equation
\begin{equation}\label{eq:HJephi pre}\tag{HJ$_e\phi$}
	\begin{cases}
		D_t u(t,x) + H(x,D_x u(t,x)) = 0, &\qquad t>0,\, x\in M, \\
		u(0,x) = \phi(x), &\qquad x\in M.
	\end{cases}
\end{equation}
We call an absolutely continuous curve $\gamma:I\to M$ a $(u,H)$-calibrated curve if
\begin{align*}
	u(t_2,\gamma(t_2))=u(t_1,\gamma(t_1))+\int_{t_1}^{t_2}L(\gamma(s),\dot{\gamma}(s))\ ds,\quad \forall [t_1,t_2]\subset I,
\end{align*}
which is equivalent to
\begin{align*}
	T_{t_2}^-\phi(\gamma(t_2))&=T_{t_1}^-\phi(\gamma(t_1))+A_{t_2-t_1}(\gamma(t_1),\gamma(t_2)),\\
	A_{t_2-t_1}(\gamma(t_1),\gamma(t_2))&=\int_{t_1}^{t_2}L(\gamma(s),\dot{\gamma}(s)).
\end{align*}
The following proposition collects the properties of calibrated curves that will be used throughout the paper.

\begin{Pro}[\protect{\cite[Chapter 6]{Cannarsa_Sinestrari_book}, \protect{\cite[Lemma 3.3]{Cannarsa_Cheng3}}}]\label{pro:cali}
	Let $M$ be a compact connected smooth manifold without boundary, $H$ a Tonelli Hamiltonian on $M$, $\phi\in\SCL(M)$, and $u$ the viscosity solution of \eqref{eq:HJephi pre}. Then the following statements hold.
	\begin{enumerate}[\rm (1)]
		\item For every $(t,x)\in(0,+\infty)\times M$ there exists a $(u,H)$-calibrated curve $\gamma:[0,t]\to M$ with $\gamma(t)=x$.
		\item Every $(u,H)$-calibrated curve $\gamma:[0,T]\to M$ is of class $C^2$, and the pair $(\gamma,p)$ with dual arc $p(s)=L_v(\gamma(s),\dot\gamma(s))$ satisfies the Hamiltonian system
		\begin{equation}
			\dot\gamma(s)=H_p(\gamma(s),p(s)),\qquad \dot p(s)=-H_x(\gamma(s),p(s)),\qquad \forall s\in[0,T].
		\end{equation}
		\item There exist constants $C_1>0$ and $C_2>0$, depending on the local $C^2$ modulus and the Tonelli modulus of $H$ and the Lipschitz and semiconcavity constants of $\phi$, such that every $(u,H)$-calibrated curve satisfies $|\dot\gamma(s)|\leqslant C_1$ and $|\ddot\gamma(s)|\leqslant C_2$ for all $s$ in its domain.
		\item For every $(u,H)$-calibrated curve $\gamma:[0,T]\to M$ and every $s\in[0,T)$, the function $u(s,\cdot)$ is differentiable at $\gamma(s)$. Moreover, we have $p(s)=L_v(\gamma(s),\dot\gamma(s))=D_xu(s,\gamma(s))$ and $p(T)\in D^+_xu(T,\gamma(T))$.
	\end{enumerate}
\end{Pro}

If $\phi$ is a viscosity solution of the stationary equation \eqref{eq:HJs}, we also call such $\gamma$ a $(\phi,H)$-calibrated curve.  For $(t,x)\in(0,\infty)\times M$ we denote by $D^*_xu(t,x)$ the \emph{calibrated subdifferential} of $u(t,\cdot)$ at $x$, i.e.\ the set of momenta $p=L_v(\gamma(t),\dot\gamma(t))$ of $(u,H)$-calibrated curves $\gamma$ with $\gamma(t)=x$. It is known that $D^+_xu(t,x)=\operatorname{co}D^*_xu(t,x)$ (see \cite[Chapter 6]{Cannarsa_Sinestrari_book}).

We define the set of \emph{regular points} and \emph{cut points} for the pair $(u,H)$ as
\begin{align*}
	\reg(u,H)=&\ \bigl\{\ (t,x)\in [0,+\infty)\times M:\exists\ \delta>0\text{ and a $(u,H)$-calibrated curve}\\
	&\ \gamma:[t,t+\delta]\to M\ \text{with}\ \gamma(t) = x\ \bigr\},\\
	\cut(u,H)=&\ \big([0,+\infty)\times M\big)\setminus\reg(u,H).
\end{align*}

\begin{defn}[Cut time function]\label{defn:tau_phi}
	Let $\phi\in \text{\rm SCL}(M)$, $H:T^*M\to\R$ be a Tonelli Hamiltonian, and $u$ be the viscosity solution of Cauchy problem \eqref{eq:HJephi pre}. We define the \emph{cut time function} of $\phi$ as (see also \cite{Cannarsa_Cheng_Hong2025})
	\begin{align*}
		\tau_{\phi,H}(x) = \sup\bigl\{\ t\geqslant0: \exists\  \text{a $(u,H)$-calibrated curve}\ \gamma:[0,t]\to M\ \text{with}\ \gamma(0) = x\ \bigr\},\quad x\in M,
	\end{align*}
	 and the sup-level set of $\tau_{\phi,H}$ as
	\begin{align*}
		E(\phi,H,t)=\{x\in M: \tau_{\phi,H}(x)\geqslant t\ \},\qquad t\geqslant0.\ 
	\end{align*}
	We then introduce the sets
	\begin{align*}
		\reg(\phi,H)=\bigl\{\ x\in M:\tau_{\phi,H}(x)>0\ \bigr\}, \quad
		\cut(\phi,H)=\bigl\{\ x\in M:\tau_{\phi,H}(x)=0\ \bigr\},
	\end{align*} 
	which are called the set of \emph{regular points} and \emph{cut points} for the pair $(\phi,H)$, respectively.
\end{defn}
From the definitions above, we have the relation
\begin{equation}\label{eq:cut u rel}
	\reg(u,H)=\bigcup_{t\geqslant0}\{t\}\times\reg(u(t,\cdot),H),\qquad\cut(u,H)=\bigcup_{t\geqslant0}\{t\}\times\cut(u(t,\cdot),H).
\end{equation}

We shall use the following elementary uniqueness fact for calibrated curves from regular points.

\begin{Lem}\label{lem:uniq cali}
	Let $M$ be a compact connected smooth manifold without boundary, $H$ a Tonelli Hamiltonian on $M$, $\phi\in\SCL(M)$, and $u$ the viscosity solution of \eqref{eq:HJephi pre}. If $x\in\reg(\phi,H)$ and $t\in(0,\tau_{\phi,H}(x)]$ for $\tau_{\phi,H}(x)<+\infty$, or $t\in(0,+\infty)$ for $\tau_{\phi,H}(x)=+\infty$, then there exists a unique $(u,H)$-calibrated curve $\gamma:[0,t]\to M$ with $\gamma(0)=x$, and it is given by
	\begin{align*}
		\gamma(s)=\pi\circ\Phi^s_H(x,D\phi(x)),\qquad \forall s\in[0,t],
	\end{align*}
	where $\pi:T^*M\to M$ is the canonical projection and $\Phi^s_H$ is the Hamiltonian flow of $H$.
\end{Lem}

\begin{proof}
	Existence for $t<\tau_{\phi,H}(x)$ holds by Definition \ref{defn:tau_phi}. Let $\gamma:[0,t]\to M$ be any $(u,H)$-calibrated curve with $\gamma(0)=x$. By Proposition \ref{pro:cali} (2) and (4), the pair $(\gamma,p)$ with dual arc $p(s)=L_v(\gamma(s),\dot\gamma(s))$ solves the Hamiltonian system and satisfies $p(0)=D\phi(x)$. Hence $\gamma(s)=\pi\circ\Phi^s_H(x,D\phi(x))$ by uniqueness of the Hamiltonian flow.
	
	It remains to treat $t=\tau_{\phi,H}(x)<+\infty$. Choose increasing times $t_k\uparrow t$ and, for each $k$, a $(u,H)$-calibrated curve on $[0,t_k]$ with initial point $x$. By what we have just proved, each of these curves coincides with the trajectory $\gamma(s)=\pi\circ\Phi^s_H(x,D\phi(x))$. This trajectory extends to the limiting time $t$, and the calibration identity
		\[
		u(t_k,\gamma(t_k))-\phi(x)=\int_0^{t_k}L(\gamma(s),\dot\gamma(s))\,ds
		\]
		passes to the limit by continuity, since $u$ is continuous on $[0,+\infty)\times M$ and $\gamma$ is of class $C^2$ on $[0,t]$. Thus a positive finite cut time is attained, the extended curve is calibrated on $[0,t]$, and uniqueness on $[0,t]$ follows by restriction to $[0,t_k]$. This completes the proof.
\end{proof}

It is well known that if $\phi$ is a viscosity solution of \eqref{eq:HJs}, $\tau_{\phi,H}$ is upper-semicontinuous (\cite{Cannarsa_Cheng_Fathi2021}). We now extend this result to any semiconcave function with linear modulus.
 
\begin{Pro}\label{pro:tau-usc}
	Let $\phi\in\SCL(M)$ and $H:T^*M\to\R$ be a Tonelli Hamiltonian. Then
	\begin{enumerate}[\rm (1)]
		\item $E(\phi,H,t)$ is compact for any $t\geqslant0$.
		\item $\tau_{\phi,H}$ is upper semicontinuous on $M$.
	\end{enumerate}
\end{Pro}

\begin{proof}
	$E(\phi,H,0)=M$ is compact. For $t>0$, we consider any sequence $\{x_k\}\subset E(\phi,H,t)$ such that $x_k\to x$ in $M$. For each $k$, pick a $(u,H)$-calibrated curve $\gamma_k:[0,t]\to M$ with $\gamma_k(0)=x_k$. By Proposition~\ref{pro:cali} (3), the curves $\gamma_k$ satisfy the uniform bound $|\ddot{\gamma}_k|\leqslant C$, so $\{(\gamma_k,\dot{\gamma}_k)\}$ is equi-Lipschitz.  By the Arzel\`a--Ascoli Theorem, a further subsequence, which we also denote by $\{(\gamma_k,\dot{\gamma}_k)\}$, converges uniformly on $[0,t]$ to a curve $\gamma:[0,t]\to M$ with $\gamma(0)=x$. Each $\gamma_k$ satisfies
	\begin{align*}
		u(t,\gamma_k(t))-\phi(x_k)=\int_0^t L(\gamma_k,\dot\gamma_k)\,ds.
	\end{align*}
	Passing to the limit and using the uniform convergence $(\gamma_k,\dot{\gamma}_k)\to(\gamma,\dot{\gamma})$, we obtain
	\begin{align*}
		u(t,\gamma(t))-\phi(x)=\int_0^t L(\gamma,\dot\gamma)\,ds.
	\end{align*}
	Hence $\gamma$ is $(u,H)$-calibrated on $[0,t]$, and $\tau_{\phi,H}(x)\geqslant t$, that is, $x\in E(\phi,H,t)$. Therefore, $E(\phi,H,t)$ is compact. This completes the proof of (1). Assertion (2) follows directly from (1).
\end{proof}

\section{Cut locus via touching approach}

\subsection{Cut locus of semiconcave functions}

Let $\Omega\subset\R^n$ be an open and convex subset. We first give the definition of regular points and cut points of semiconcave functions with linear modulus on $\Omega$, by an intuitive touching approach.

\begin{defn}\label{defn:various_sets}
	Let \( \Omega \subset \R^n \) be an open convex subset and $\phi\in\text{\rm SCL}\,(\Omega)$.
	\begin{enumerate}[(1)]

	\item We call $x\in\Omega$ a \emph{regular point} of $\phi$ if there exists a convex open neighborhood $V$ of $x$ and a semiconvex function $\psi: V \to \R$ with linear modulus touching $\phi$ from below at $x$ in $V$. We denote by $\reg(\phi)$ the set of regular points of $\phi$ and call it the \emph{regular set} of \( \phi \).
		
	\item We call $x\in\Omega$ a \emph{cut point} of $\phi$ if $x$ is not a regular point of $\phi$. We denote by $\cut(\phi)$ the set of cut points of $\phi$ and call it the \emph{cut locus} of $\phi$, that is, 
	\begin{align*}
	\cut(\phi)= \Omega \setminus \reg(\phi).
	\end{align*}
	\end{enumerate}
\end{defn}

\begin{Lem}\label{lem:smooth touching}
Let \( \Omega \subset \R^n \) be an open convex subset and $\phi\in\text{\rm SCL}\,(\Omega)$. If $x\in\reg(\phi)$, then \( \phi \) is differentiable at \( x \) and there exists an open neighborhood \( V \) of \( x \) and a \( C^\infty \) function
\begin{align*}
	\psi(y)=\phi(x)+D\phi(x)\cdot(y-x)-\frac{C}{2}|y-x|^2,\quad y\in V,
\end{align*}
where $C>0$ is a constant, such that
\begin{align*}
	\psi(y)<\phi(y), \quad \forall y\in V\setminus\{x\}, \quad \text{and} \quad \psi(x) = \phi(x).
\end{align*}
\end{Lem}

\begin{proof}
Given \( x \in \reg(\phi) \), there exists an open neighborhood \( V \) of \( x \) and a semiconvex function \( \tilde{\psi}:V\to\R \) with linear modulus \( C_1>0 \) such that $\tilde{\psi}\leqslant\phi$ on $V$ and $\tilde{\psi}(x) = \phi(x)$. Invoking Proposition \ref{pro:touching}, we conclude that both \( \phi \) and \( \tilde{\psi} \) are differentiable at \( x \) and satisfy \( D\phi(x) = D\tilde{\psi}(x) \). Now, let $C=C_1+1$. By Proposition \ref{pro:scl fund}, we have
\begin{align*}
	\psi(y):&=\phi(x)+D\phi(x)\cdot(y-x)-\frac{C}{2}|y-x|^2\\
	&= \tilde{\psi}(x) + D\tilde{\psi}(x) \cdot (y-x) - \frac{C_1+1}{2} |y-x|^2<\tilde{\psi}(y) \leqslant \phi(y), \quad \forall y \in V\setminus\{x\},
\end{align*}
and $\psi(x)=\phi(x)$.
\end{proof}

We recall a standard notion from nonsmooth analysis. For $\phi\in\text{\rm SCL}\,(\Omega)$, the \emph{proximal subdifferential} of $\phi$ is defined by
\begin{align*}
	\partial_P \phi(x)=\big\{p\in\R^n:\exists\ C,r>0\ \text{s.t.}\ \phi(y)\geqslant\phi(x)+\langle p,y-x\rangle-\frac{C}{2}|y-x|^2,\ \forall y\in B(x,r)\big\},\quad x\in\Omega.
\end{align*}

The following Proposition \ref{pro:prox} implies a point is regular if and only if the proximal subdifferential of $\phi$ at that point is nonempty. We use the touching terminology in Definition \ref{defn:various_sets} to emphasize the connection with semiconvex lower supports and the Lax--Oleinik operators.
\begin{Pro}\label{pro:prox}
	Let \( \Omega \subset \R^n \) be an open convex subset and $\phi\in\text{\rm SCL}\,(\Omega)$. Then
	\begin{align*}
		\reg(\phi)=\big\{x\in\Omega:\partial_P \phi(x)\neq\varnothing\big\},\qquad \cut(\phi)=\big\{x\in\Omega:\partial_P \phi(x)=\varnothing\big\}.
	\end{align*}
	Moreover, for any $x\in\reg(\phi)$, $\phi$ is differentiable at $x$ and $\partial_P \phi(x)=\{D\phi(x)\}$.
\end{Pro}

\begin{proof}
	For $x\in\reg(\phi)$, Lemma \ref{lem:smooth touching} implies $\phi$ is differentiable at $x$ and $D\phi(x)\in\partial_P \phi(x)$. If there exists another $p\in\partial_P \phi(x)$, by Proposition \ref{pro:touching}, we have $p=D\phi(x)$. Thus $\partial_P \phi(x)=\{D\phi(x)\}$. For $x\in\Omega$ with $\partial_P \phi(x)\neq\varnothing$, we obtain $x\in\reg(\phi)$ directly from Definition \ref{defn:various_sets}. Hence $\reg(\phi)=\big\{x\in\Omega:\partial_P \phi(x)\neq\varnothing\big\}$, and consequently $\cut(\phi)=\big\{x\in\Omega:\partial_P \phi(x)=\varnothing\big\}$. This completes the proof.
\end{proof}

The following Proposition \ref{pro:touching_diff_reg} gives three equivalent characterizations of regular points, in terms of touching, linear estimation, and the $C^{1,1}$ property, respectively.
\begin{Pro}\label{pro:touching_diff_reg}
	Let \( \Omega \subset \R^n \) be an open convex subset and $\phi\in\text{\rm SCL}\,(\Omega)$. For any $x\in\Omega$, the following statements are equivalent.
	\begin{enumerate}[(i)]
		\item $x\in\reg(\phi)$.
		\item $\phi$ is differentiable at $x$, and there exists an open neighborhood \( V \) of \( x \) and $C>0$ such that
		\begin{equation}\label{eq:cut equi 1}
			|\phi(y)-\phi(x)-D\phi(x)\cdot(y-x)|\leqslant \frac{C}{2}|y-x|^2,\quad \forall y\in V.
		\end{equation}
		\item $\phi$ is differentiable at $x$, and there exists an open neighborhood \( V \) of \( x \) and $\lambda>0$ such that
		\begin{align*}
			|p-D\phi(x)|\leqslant\lambda|y-x|,\quad \forall y\in V,\ p\in D^+\phi(y).
		\end{align*}
	\end{enumerate}
\end{Pro}

\begin{proof}
	(i) $\Rightarrow$ (ii): On the one hand, Lemma \ref{lem:smooth touching} implies $\phi$ is differentiable at $x$ and there exists an open neighborhood $V$ of $x$ and $C_1>0$ such that
	\begin{equation}\label{eq:cut equi pf1}
		\phi(y)\geqslant\phi(x)+D\phi(x)\cdot(y-x)-\frac{C_1}{2}|y-x|^2,\quad \forall y\in V.
	\end{equation}
	On the other hand, $\phi\in\text{\rm SCL}\,(\Omega)$ implies there exists $C_2>0$ such that, by Proposition \ref{pro:scl fund},
	\begin{equation}\label{eq:cut equi pf2}
		\phi(y)\leqslant\phi(x)+D\phi(x)\cdot(y-x)+\frac{C_2}{2}|y-x|^2,\quad \forall y\in\Omega.
	\end{equation}
	Let $C=\max\{C_1,C_2\}$. Then \eqref{eq:cut equi 1} follows from \eqref{eq:cut equi pf1} and \eqref{eq:cut equi pf2}.
	
	(ii) $\Rightarrow$ (i): We let the function
	\begin{align*}
		\psi(y)=\phi(x)+D\phi(x)\cdot(y-x)-\frac{C}{2}|y-x|^2,\quad y\in V.
	\end{align*}
	Then $\psi$ is of class $C^{\infty}$, and \eqref{eq:cut equi 1} implies $\psi\leqslant\phi$ on $V$ and $\psi(x)=\phi(x)$.
	
	(i) $\Leftrightarrow$ (iii): It is a direct consequence of the following Lemma \ref{lem:touching_C11} and Lemma \ref{lem:touching_2}.
\end{proof}

\begin{Lem}\label{lem:touching_C11}
	Let $x\in\R^n$, $r>0$ and $\phi\in\text{\rm SCL}\,(B(x,r))$ with linear modulus $C_1>0$, and let $\psi:B(x,r)\to\R$ be a semiconvex function with linear modulus $C_2>0$. If $\psi$ touches $\phi$ from below at $x$ in $B(x,r)$, then we have the following properties:
	\begin{enumerate}[\rm (1)]
		\item both $\phi$ and $\psi$ are differentiable at $x$ and $D\psi(x)=D\phi(x)$.
		\item for $C=\max\{C_1,C_2\}$ we have that
		\begin{align*}
			|p-D\phi(x)|\leqslant 3C|y-x|,\qquad \forall y\in B(x,\frac r2), p\in D^+\phi(y);\\
			|p-D\psi(x)|\leqslant 3C|y-x|,\qquad \forall y\in B(x,\frac r2), p\in D^-\psi(y).
		\end{align*}
	\end{enumerate}
\end{Lem}

\begin{proof}
	Proposition \ref{pro:touching} implies both \( \phi \) and \( \psi \) are differentiable at \( x \) and \( D\phi(x) = D\psi(x) \). Now, we turn to the proof of (2). Fix $y\in B(x,\frac r2)$, $p\in D^+\phi(y)$, and $\theta\in\R^n$ with $|\theta|<\frac r2$. By the $C_1$-semiconcavity of $\phi$ and Proposition \ref{pro:scl fund}, we have
	\begin{align*}
		\phi(y)\leqslant&\,\phi(x)+D\phi(x)\cdot(y-x)+\frac{C_1}2|y-x|^2,\\
		\phi(y+\theta)\leqslant&\,\phi(y)+p\cdot\theta+\frac{C_1}2|\theta|^2.
	\end{align*}
	The $C_2$-semiconvexity of $\psi$, Proposition \ref{pro:scl fund} and the touching assumption yields
	\begin{align*}
		D\phi(x)\cdot(y+\theta-x)-\frac{C_2}2|y+\theta-x|^2\leqslant\psi(y+\theta)-\psi(x)\leqslant\phi(y+\theta)-\phi(x).
	\end{align*}
	Combining the inequalities above we obtain
	\begin{equation}\label{eq:c11 pf}
		(D\phi(x)-p)\cdot\theta\leqslant\frac{C_1}2|y-x|^2+\frac{C_1}2|\theta|^2+\frac{C_2}2|y+\theta-x|^2.
	\end{equation}
    Without loss of generality we assume $p\not=D\phi(x)$. Then for $\theta=\frac{|y-x|}{|D\phi(x)-p|}(D\phi(x)-p)$ in \eqref{eq:c11 pf}, we conclude that
	\begin{align*}
		|D\phi(x)-p|\cdot|y-x|\leqslant&\,\frac{C_1}2|y-x|^2+\frac{C_1}2|y-x|^2+\frac{C_2}2\big|y-x+\frac{|y-x|}{|D\phi(x)-p|}(D\phi(x)-p)\big|^2\\
		\leqslant&\,3C|y-x|^2.
	\end{align*}
	Observe that $y\not=x$ since $p\not=D\phi(x)$. This leads to the first inequality in (2). The second one follows by the same argument with the roles of $\phi$ and $\psi$ exchanged.
\end{proof}

\begin{Lem}\label{lem:touching_2}
	Let \( \Omega \subset \R^n \) be an open convex subset and $\phi\in\text{\rm SCL}\,(\Omega)$. If $\phi$ is differentiable at $x\in\Omega$ and there exists $\lambda>0$ such that
	\begin{align*}
		|p-D\phi(x)|\leqslant \lambda|y-x|,\qquad\forall y\in\Omega, p\in D^+\phi(y).
	\end{align*}
	Then the function
	\begin{align*}
		\psi(y)=\phi(x)+D\phi(x)\cdot(y-x)-\frac{\lambda}{2}|y-x|^2,\qquad y\in\Omega,
	\end{align*}
	is of class $C^{\infty}$ and touches $\phi$ from below at $x$ in $\Omega$.
\end{Lem}

\begin{proof}
    Obviously, $\psi$ is of class $C^{\infty}$ and $\phi(x)=\psi(x)$. Given $y\in\Omega$. Due to the integral representation of semiconcave functions in terms of the superdifferential (see \cite[Chapter 3]{Cannarsa_Sinestrari_book}) and our assumption, it yields
\begin{align*}
	\phi(y)=&\,\phi(x)+\int^1_0\min_{p\in D^+\phi(x+r(y-x))}p\cdot(y-x)\ dr\\
	=&\,\phi(x)+D\phi(x)\cdot(y-x)+\int^1_0\min_{p\in D^+\phi(x+r(y-x))}(p-D\phi(x))\cdot(y-x)\ dr\\
	\geqslant&\,\phi(x)+D\phi(x)\cdot(y-x)-\int^1_0\lambda\cdot|x+r(y-x)-x|\cdot|y-x|\ dr\\
	=&\,\phi(x)+D\phi(x)\cdot(y-x)-\frac \lambda2|y-x|^2=\psi(y).
\end{align*}\qedhere
\end{proof}

Now we have the following relations 
\begin{Pro}\label{pro:singcutalex}
	Let \( \Omega \subset \R^n \) be an open convex subset and $\phi\in\text{\rm SCL}\,(\Omega)$. Then we have
	\begin{align*} \text{\rm Alex}\,(\phi)\subset\reg(\phi)\subset\text{\rm Diff}\,(\phi).
	\end{align*}
\end{Pro}

\begin{proof}
For any $x\in\text{Alex}\,(\phi)$, it obviously satisfies (ii) in Proposition \ref{pro:touching_diff_reg}. This implies $\text{Alex}\,(\phi)\subset\reg(\phi)$. $\reg(\phi)\subset\text{Diff}\,(\phi)$ follows directly from Proposition \ref{pro:touching}.
\end{proof}

We can see that the relations in Proposition \ref{pro:singcutalex} are nontrivial from the following two examples. In particular, Example \ref{ex:singcut} gives a differentiable cut point.

\begin{Ex}\label{ex:singcut}
	Let
	\begin{align*}
		\phi(x)=-|x|^{\frac{3}{2}},\qquad x\in\R.
	\end{align*}
	Then $\phi\in\text{\rm SCL}\,(\R)$ since $\phi$ is obviously a concave function. Now we have
	\begin{align*}
		D\phi(x)=
		\begin{cases}
			\ \frac{3}{2}|x|^{\frac{1}{2}},\quad &x\in(-\infty,0),\\
			\ -\frac{3}{2}|x|^{\frac{1}{2}},\quad &x\in[0,+\infty).
		\end{cases}
	\end{align*}
	Thus, $\text{Diff}\,(\phi)=\R$, and Proposition \ref{pro:touching_diff_reg} implies $\reg(\phi)=\R\setminus\{0\}$.
\end{Ex}

\begin{Ex}
	Let
	\begin{align*}
		\phi(x)=
		\begin{cases}
			\ -\frac{1}{2}x^2,\qquad &x\in(-\infty,0),\\
			\ \frac{1}{2}x^2,\qquad &x\in[0,+\infty).
		\end{cases}
	\end{align*}
	Then
	\begin{align*}
		D\phi(x)=
		\begin{cases}
			\ -x,\quad &x\in(-\infty,0),\\
			\ x,\quad &x\in[0,+\infty),
		\end{cases}
		\qquad
		D^2\phi(x)=
		\begin{cases}
			\ -1,\quad &x\in(-\infty,0),\\
			\ 1,\quad &x\in(0,+\infty),
		\end{cases}
	\end{align*}
	and $\phi$ is not twice differentiable at $0$. Thus, $\phi\in C^{1,1}(\R)\subset\text{\rm SCL}\,(\R)$, and it is a viscosity solution of the Hamilton--Jacobi equation
	\begin{align*}
		|D\phi(x)|^2-x^2=0,\qquad x\in\R.
	\end{align*}
	Now we have $\text{Alex}\,(\phi)=\R\setminus\{0\}$, and Proposition \ref{pro:touching_diff_reg} implies $\reg(\phi)=\R$. 
\end{Ex}

\begin{Rem}
We remark that all the statements above in this subsection hold true for a smooth manifold $M$ without boundary, and a locally semiconcave function with linear modulus on $M$, since all the definitions and estimates are local in nature.
\end{Rem}

\subsection{Degree of regularity of semiconcave functions}
We now work on the $n$-dimensional flat torus $\T^n$, for brevity and to highlight the main ideas. All the results below extend to general compact manifolds without boundary by standard localization arguments.


Inspired by Proposition \ref{pro:touching_diff_reg}, we can have a finer characterization of $\reg(\phi)$ introduced previously for $\phi\in\text{\rm SCL}\,(\T^n)$.

\begin{Lem}\label{pro:reform_reg}
	Let $\phi\in\text{\rm SCL}\,(\T^n)$. Then
	\begin{align*}
		\reg(\phi)=\{x\in\T^n:D\phi(x) \text{ exists and } \exists\ \lambda_x\geqslant0\ \text{s.t.}\ |p-D\phi(x)|\leqslant\lambda_x\cdot  d(y,x),\forall y\in\T^n, p\in D^+\phi(y)\}.
	\end{align*}
\end{Lem}

\begin{proof}
	In view of Proposition \ref{pro:touching_diff_reg}, the set on the right-side of the equality above is contained in $\reg(\phi)$. For the opposite direction let $x\in\reg(\phi)$. Thanks again to Proposition \ref{pro:touching_diff_reg}, $D\phi(x)$ exists and there exists $0<r<\frac 1{10}$ and $\lambda_{x,1}\geqslant0$ such that
	\begin{align*}
		|p-D\phi(x)|\leqslant\lambda_{x,1}|y-x|,\qquad\forall y\in B(x,r), p\in D^+\phi(y).
	\end{align*}
	We set
	\begin{align*}
		\lambda_{x,2}&=\sup\bigg\{\frac{|p-D\phi(x)|}{d(y,x)}:\ y\in\T^n\setminus B(x,r),\\
		&\qquad\qquad p\in D^+\phi(y)\bigg\}.
	\end{align*}
	Since $\phi$ is Lipschitz, $D^+\phi$ is uniformly bounded, hence $\lambda_{x,2}<+\infty$. Let $\lambda_x=\max\{\lambda_{x,1},\lambda_{x,2}\}$. Then we obtain the following inequality
	\begin{align*}
		|p-D\phi(x)|\leqslant\lambda_x\cdot d(y,x),\qquad \forall y\in\T^n, p\in D^+\phi(y).
	\end{align*}
	This completes the proof.
\end{proof}

\begin{defn}\label{def:R_phi}
	Let $\phi\in\text{\rm SCL}\,(\T^n)$. We define the function of \emph{degree of regularity} of $\phi$ as
	\begin{align*}
		R_\phi(x)=
		\begin{cases}
			\inf\{\lambda\geqslant0: |p-D\phi(x)|\leqslant\lambda\cdot d(y,x),\forall y\in\T^n, p\in D^+\phi(y)\},&x\in \reg(\phi),\\
			+\infty,&x\in\cut(\phi),
		\end{cases}
	\end{align*}
	and the sub-level set of $R_\phi$ as
	\begin{align*}
		F(\phi,\lambda)=\{x\in\T^n: R_\phi(x)\leqslant\lambda\},\qquad\lambda\geqslant0.
	\end{align*}
\end{defn}

\begin{Rem}
	In general, the $C^{1,1}$ property of regular points is local (Proposition \ref{pro:touching_diff_reg}). Here we can define $R_\phi(x)$ by imposing the estimate for every $y\in\T^n$ since $\T^n$ is compact (Lemma \ref{pro:reform_reg}). Thus, $R_\phi$ is a globally normalized regularity
	defect, rather than a purely local $C^{1,1}$ deviation. This is consistent with the truncation constants in Theorem \ref{thm:1} below.
\end{Rem}

\begin{Pro}\label{lem:R_phi}
	Let $\phi\in\text{\rm SCL}\,(\T^n)$. Then we have the following statements.
	\begin{enumerate}[\rm (1)]
		\item $\reg(\phi)=\{x\in\T^n: R_\phi(x)<+\infty\}$ and $\cut(\phi)=\{x\in\T^n: R_\phi(x)=+\infty\}$.
		\item $F(\phi,\lambda)$ is compact for any $\lambda\geqslant0$.
		\item $R_\phi$ is lower semi-continuous on $\T^n$.
	\end{enumerate}
\end{Pro}

\begin{proof}
	Assertion (1) follows directly by the definition of $R_\phi$ and Lemma \ref{pro:reform_reg}.
	
	Now we turn to prove (2). For $\lambda\geqslant0$, choose a sequence $\{x_i\}\subset F(\phi,\lambda)$ such that $\lim_{i\to\infty}x_i=x\in\T^n$. Then
	\begin{align*}
		|p-D\phi(x_i)|\leqslant\lambda\cdot d(y,x_i),\qquad\forall y\in\T^n, p\in D^+\phi(y), i\in\N.
	\end{align*}
	Since $D^+\phi$ is uniformly bounded and upper-semicontinuous, we can take a subsequence $\{x_{i_k}\}$ such that $\lim_{k\to\infty}D\phi(x_{i_k})=q\in D^+\phi(x)$. Now we have
	\begin{align*}
		|p-q|\leqslant\lambda\cdot d(y,x),\qquad\forall y\in\T^n, p\in D^+\phi(y).
	\end{align*}
	For $y=x$, from the inequality above we conclude that $D^+\phi(x)=\{q\}$ is a singleton, i.e., $\phi$ is differentiable at $x$ and $D\phi(x)=q$. It follows that
	\begin{align*}
		|p-D\phi(x)|\leqslant\lambda\cdot d(y,x),\qquad\forall y\in\T^n, p\in D^+\phi(y).
	\end{align*}
	In other words, $x\in F(\phi,\lambda)$. This implies $F(\phi,\lambda)$ is compact, and completes the proof of (2). Assertion (3) is a direct consequence of (2).
\end{proof}
 

The following Theorem \ref{thm:1} clarifies the relation between the degree of regularity $R_\phi$ (Definition~\ref{def:R_phi}) and the cut time function $\tau_{\phi,H}$ (Definition~\ref{defn:tau_phi}).

\begin{The}\label{thm:1}
	Let $H:\T^n\times\R^n\to\R$ be a Tonelli Hamiltonian and
	$\phi\in\SCL(\T^n)$. Then there exist positive constants
	$C_{1,1},C_{1,2},C_{2,1},C_{2,2}$
	(depending on the local $C^2$ modulus and the Tonelli modulus of $H$ and the Lipschitz and semiconcavity constants of $\phi$) such that
	\begin{enumerate}[\rm (1)]
		\item $R_\phi(x)\leqslant\max\Big\{\, \frac{C_{1,1}}{\tau_{\phi,H}(x)},\ C_{1,2}\ \Big\},\quad \forall x\in \T^n,$
		\item $\tau_{\phi,H}(x) \geqslant \min\Big\{\, \frac{C_{2,1}}{R_\phi(x)},\ C_{2,2}\ \Big\},\quad \forall x\in \T^n,$
	\end{enumerate}
where we let $\frac{1}{0}=+\infty$ and $\frac{1}{+\infty}=0$.
\end{The}


\begin{proof}
	(1) Let $S\geqslant0$ be a semiconcavity constant of $\phi$ and let $P$ be a Lipschitz constant of $\phi$ on $\T^n$. Let $V>0$ be a uniform speed bound for all $(u,H)$-calibrated curves, supplied by Proposition~\ref{pro:cali} (3). Fix $R=V+2$ and obtain $K_R,t_R$ from Proposition~\ref{pro:At} (1). Choose a fixed $t_*\in(0,\min\{t_R,1\}]$ small enough for all the coordinate balls used below to be valid. These constants are independent of $x_0$.

Suppose $\tau_{\phi,H}(x_0)>0$, choose
\[
0<t<\min\{\tau_{\phi,H}(x_0),t_*\},
\]
and let $\gamma:[0,t]\to\T^n$ be a $(u,H)$-calibrated curve with $\gamma(0)=x_0$. Set $y_t=\gamma(t)$. Since $d(x_0,y_t)\leqslant Vt$, we have
\[
B(x_0,t)\subset B(y_t,Rt).
\]
Calibration gives
\[
\phi(x_0)+A_t(x_0,y_t)=\min_{z\in\T^n}\{\phi(z)+A_t(z,y_t)\}.
\]
Thus the function
\[
\psi_t(z)=\phi(x_0)+A_t(x_0,y_t)-A_t(z,y_t)
\]
touches $\phi$ from below at $x_0$. On $B(x_0,t)$ it is semiconvex with constant $K_R/t$. Lemma~\ref{lem:touching_C11} therefore yields
\begin{equation}\label{eq:near}
|p-D\phi(x_0)|\leqslant3\max\{S,K_R/t\}\,d(z,x_0),\qquad \forall z\in B(x_0,t/2),\ p\in D^+\phi(z).
\end{equation}
For $z\notin B(x_0,t/2)$, boundedness of the superdifferential gives
\begin{equation}\label{eq:far}
|p-D\phi(x_0)|\leqslant2P\leqslant\frac{4P}{t}\,d(z,x_0),\qquad \forall p\in D^+\phi(z).
\end{equation}
Taking
\[
C=\max\{1,3St_*,3K_R,4P\},
\]
we conclude from \eqref{eq:near}--\eqref{eq:far} that
\[
R_\phi(x_0)\leqslant\frac Ct,\qquad \forall 0<t<\min\{\tau_{\phi,H}(x_0),t_*\}.
\]
Letting $t\uparrow\min\{\tau_{\phi,H}(x_0),t_*\}$ gives
\[
R_\phi(x_0)\leqslant\max\left\{\frac{C}{\tau_{\phi,H}(x_0)},\frac{C}{t_*}\right\}.
\]
This remains valid when $\tau_{\phi,H}(x_0)=+\infty$, with the usual convention $1/(+\infty)=0$. If $\tau_{\phi,H}(x_0)=0$, the asserted bound is automatic. Hence Theorem~\ref{thm:1} (1) follows with $C_{1,1}=C$ and $C_{1,2}=C/t_*$. This completes the proof of (1).
	
	(2) For any $x_0\in\T^n$, if $R_{\phi}(x_0)=+\infty$, then assertion (2) obviously holds. Now we consider the case $R_{\phi}(x_0)<+\infty$. Choose any $\lambda\in(R_\phi(x_0),+\infty)$. By Lemma \ref{lem:touching_2} the $C^\infty$-function
	\begin{equation}\label{eq:1}
		\psi(x)=\phi(x_0)+D\phi(x_0)\cdot(x-x_0)-\frac {\lambda}2|x-x_0|^2,\qquad x\in B(x_0,1/100),
	\end{equation}
	touches $\phi$ from below at $x_0$. It follows that
	\begin{equation}\label{eq:2}
		\phi(x_0) - \psi(x_0)=0 \leqslant \phi(x) - \psi(x), \quad  \forall x\in B(x_0,1/100).
	\end{equation}
	Define a Hamiltonian trajectory
	\begin{align*}
		(\gamma(s), p(s))=\Phi^s_H(x_0, D\phi(x_0)),\quad s\in [0,1],
	\end{align*}
	where \( \Phi^s_H \) is the Hamiltonian flow of \( H \). Set
	\begin{align*}
		r_1=\sup\{|H_p(\Phi^s_H(x,p))|: x\in\T^n, p\in D^+\phi(x), s\in[0,1]\}.
	\end{align*}
	The global Lipschitz property of $\phi$ implies $r_1<+\infty$. Then, for \( t \in [0,1] \), we have \( x_0 \in B(\gamma(t),r_1 t) \). By \cite[Proposition 5.2]{CCHW2024}, there exists \( 0 < \tau_1 \leqslant 1 \), independent of $x_0$, such that for \( 0 < t \leqslant \tau_1 \)
	\begin{align*}
		\int_0^t L(\gamma(s), \dot{\gamma}(s))\, ds = A_t(x_0, \gamma(t)),
	\end{align*}
	and the semiconvex function $A_t(\cdot, \gamma(t))$ is differentiable at $x_0$ and
	\begin{equation} \label{eq:3}
		D_x A_t(x_0, \gamma(t)) = -D\phi(x_0).
	\end{equation}
	By Proposition~\ref{pro:cali} (3), there exists $r_2>0$ such that
	\begin{equation}\label{eq:4}
		\arg\min_{x\in\T^n}\{\phi(x)+A_t(x,y)\}\subset B(y,r_2t),\qquad\forall t>0, y\in\T^n.
	\end{equation}
	Indeed, if $\hat{x}$ is a minimizer of $x\mapsto\phi(x)+A_t(x,y)$, then any $A_t$-minimizing curve from $\hat{x}$ to $y$ is a $(u,H)$-calibrated curve, so that $d(\hat{x},y)\leqslant C_1 t$ by Proposition~\ref{pro:cali} (3).
	Take $r_3=\max\{r_1,r_2\}$. By Proposition~\ref{pro:At} (2) with $R=r_3$, there exist $\tau_2\in(0,1]$ and $K_4>0$ such that for any $t\in(0,\tau_2]$ and $y\in\T^n$, the function $x\mapsto A_t(x,y)$ is of class $C^2$ on $B(y,r_3t)$ and satisfies
	\begin{equation}\label{eq:5}
		D_{xx}A_t(x,y)\geqslant\frac{K_4}tId,\qquad\forall x\in B(y,r_3t).
	\end{equation}
	Now setting
	\begin{equation}\label{eq:6}
		C_{2,1}=K_4,\quad C_{2,2}=\min\{\tau_1,\tau_2,\frac 1{200r_3}\},\quad \tau=\min\{\frac{C_{2,1}}{\lambda},C_{2,2}\},
	\end{equation}
	and combining \eqref{eq:1}, \eqref{eq:2}, \eqref{eq:3}, \eqref{eq:5} and \eqref{eq:6}, we conclude that for any $x\in B(\gamma(\tau),r_3\tau)$
	\begin{align*}
		\phi(x_0)-\phi(x)\leqslant&\,\phi(x_0)-\psi(x)=-D\phi(x_0)\cdot(x-x_0)+\frac{\lambda}{2}|x-x_0|^2\\
		=&\,D_xA_\tau(x_0,\gamma(\tau))\cdot(x-x_0)+\frac{\lambda}{2}|x-x_0|^2\\
		\leqslant&\,D_xA_\tau(x_0,\gamma(\tau))\cdot(x-x_0)+\frac {K_4}{2\tau}|x-x_0|^2\\
		\leqslant&\,A_\tau(x,\gamma(\tau))-A_\tau(x_0,\gamma(\tau)).
	\end{align*}
	Together with \eqref{eq:4} it follows that
	\begin{align*}
		\phi(x_0)+A_\tau(x_0,\gamma(\tau))=\inf_{x\in\T^n}\{\phi(x)+A_\tau(x,\gamma(\tau))\}.
	\end{align*}
	This implies $\tau_{\phi,H}(x_0)\geqslant\tau=\min\{\frac{C_{2,1}}{\lambda},C_{2,2}\}$. Let $\lambda\to R_\phi(x_0)$, we have $\tau_{\phi,H}(x_0)\geqslant\min\{\frac{C_{2,1}}{R_\phi(x_0)},C_{2,2}\}$. This completes the proof of (2).
\end{proof}

The following Theorem \ref{thm:Reg_Cut} shows that the two definitions of regular point and cut point, via the touching approach and the variational approach, are equivalent. In other words, the cut locus is an analytic notion associated to the semiconcave function itself, independent of the choice of the Hamiltonian.

\begin{The}\label{thm:Reg_Cut}
Let \( \phi\in\text{\rm SCL}(\T^n) \) and $H:\T^n\times\R^n\to\R$ be a Tonelli Hamiltonian. Then we have
\begin{align*}
	\reg(\phi,H) = \reg(\phi), \qquad
	\cut(\phi,H) = \cut(\phi).
\end{align*}
In other words,
\begin{align*}
	\reg(\phi,H)=\big\{x\in\T^n:\partial_P \phi(x)\neq\varnothing\big\},\qquad \cut(\phi,H)=\big\{x\in\T^n:\partial_P \phi(x)=\varnothing\big\}.
\end{align*}
\end{The}

\begin{proof}
For any $x\in\reg(\phi,H)$, we have $\tau_{\phi,H}(x)>0$. Then Theorem \ref{thm:1} (1) leads to
\begin{align*}
	R_\phi(x)\leqslant\max\Big\{\, \frac{C_{1,1}}{\tau_{\phi,H}(x)},\ C_{1,2}\ \Big\}<+\infty,
\end{align*}
which implies $x\in\reg(\phi)$. Thus, we have $\reg(\phi,H)\subset\reg(\phi)$. On the other hand, for any $x\in\reg(\phi)$, we have $R_{\phi}(x)<+\infty$. Then Theorem \ref{thm:1} (2) leads to
\begin{align*}
	\tau_{\phi,H}(x) \geqslant \min\Big\{\, \frac{C_{2,1}}{R_\phi(x)},\ C_{2,2}\ \Big\}>0,
\end{align*}
which implies $x\in\reg(\phi,H)$. Thus, we have $\reg(\phi)\subset\reg(\phi,H)$. It follows that $\reg(\phi,H)=\reg(\phi)$, and $\cut(\phi,H)=\T^n\setminus \reg(\phi,H)=\T^n\setminus \reg(\phi)=\cut(\phi)$. Now, the remaining two relations are direct consequences of Proposition \ref{pro:prox}.
\end{proof}

\subsection{Cut locus and Alexandrov's theorem}

In this subsection we prove, by a variational method, that the cut locus is a Lebesgue null set, and we give a streamlined proof of Alexandrov's theorem along the same lines.

To apply Theorem \ref{thm:1}, we need the following Proposition \ref{pro:uni lip semi} on the uniform Lipschitz and semiconcavity of viscosity solution to the Cauchy problem \eqref{eq:HJephi pre}.
\begin{Pro}\label{pro:uni lip semi}
	Let $H:\T^n\times\R^n\to\R$ be a Tonelli Hamiltonian, $\phi\in \text{\rm SCL}(\T^n)$, and $u$ be the viscosity solution of Cauchy problem \eqref{eq:HJephi pre}. Then there exist $C_1>0$ and $C_2>0$ (depending on the local $C^2$ modulus and the Tonelli modulus of $H$ and the Lipschitz and semiconcavity constants of $\phi$) such that $u(t,\cdot)$ is $C_1$-Lipschitz and $C_2$-semiconcave for all $t\geqslant0$.
\end{Pro}
\begin{proof}
	We only prove the semiconcavity, since the Lipschitz continuity can be obtained in a similar and even simpler way. For any $t\geqslant0$ and $x\in\T^n$, there exists a $(u,H)$-calibrated curve $\gamma:[0,t]\to \T^n$ with $\gamma(t)=x$. Moreover, $\gamma$ is of class $C^2$ and $\dot{\gamma}$ is uniformly bounded by the compactness of $\T^n$ and the semiconcavity of $\phi$. We lift $\gamma$ to the universal covering space $\R^n$, and set $z\in\R^n$ with $|z|<\frac{1}{2}$. If $0\leqslant t\leqslant1$, let
	\begin{align*}
		\gamma_{\pm}(s)=\gamma(s)\pm z,\quad s\in[0,t].
	\end{align*}
	Then there exists $C_{0,1}>0$ and $C_{0,2}>0$ such that
	\begin{equation}\label{eq:pf uni 1}
		\begin{split}
			&u(t,x+z)+u(t,x-z)-2u(t,x)\\
			\leqslant&\phi(\gamma_{+}(0))+\int_{0}^{t}L(\gamma_{+},\dot{\gamma}_{+})\ ds+\phi(\gamma_{-}(0))+\int_{0}^{t}L(\gamma_{-},\dot{\gamma}_{-})\ ds-2\big(\phi(\gamma(0))+\int_{0}^{t}L(\gamma,\dot{\gamma})\ ds\big)\\
			\leqslant&\big(\phi(\gamma(0)+z)+\phi(\gamma(0)-z)-2\phi(\gamma(0))\big)+\big(\int_{0}^{t}L(\gamma+z,\dot{\gamma})+L(\gamma-z,\dot{\gamma})-2L(\gamma,\dot{\gamma})\ ds\big)\\
			\leqslant&C_{0,1}|z|^2+tC_{0,2}|z|^2\leqslant\big(C_{0,1}+C_{0,2}\big)|z|^2.
		\end{split}
	\end{equation}
	If $t>1$, let
	\begin{align*}
		\gamma_{\pm}(s)=
		\begin{cases}
			\gamma(s),\quad s\in[0,t-1],\\
			\gamma(s)\pm\big(s-(t-1)\big)z,\quad s\in(t-1,t].
		\end{cases}
	\end{align*} 
	Then there exists $C_{0,3}>0$ such that
	\begin{equation}\label{eq:pf uni 2}
		\begin{split}
			&u(t,x+z)+u(t,x-z)-2u(t,x)\\
			\leqslant&\phi(\gamma_{+}(0))+\int_{0}^{t}L(\gamma_{+},\dot{\gamma}_{+})\ ds+\phi(\gamma_{-}(0))+\int_{0}^{t}L(\gamma_{-},\dot{\gamma}_{-})\ ds-2\big(\phi(\gamma(0))+\int_{0}^{t}L(\gamma,\dot{\gamma})\ ds\big)\\
			\leqslant&\int_{t-1}^{t}L(\gamma+\big(s-(t-1)\big)z,\dot{\gamma}+z)+L(\gamma-\big(s-(t-1)\big)z,\dot{\gamma}-z)-2L(\gamma,\dot{\gamma})\ ds\\
			\leqslant&\int_{t-1}^{t}C_{0,3}\big(|\big(s-(t-1)\big)z|^2+|z|^2\big)\ ds\leqslant \frac43C_{0,3}|z|^2.
		\end{split}
	\end{equation}
	Let $C_2=\max\{C_{0,1}+C_{0,2},\frac43C_{0,3}\}$. Since $u$ is continuous, by \eqref{eq:pf uni 1} and \eqref{eq:pf uni 2}, $u(t,\cdot)$ is $C_2$-semiconcave for all $t\geqslant0$.
\end{proof}

\begin{The}\label{thm:cut0}
	Let $\phi\in\text{\rm SCL}\,(\T^n)$. Then we have
	\begin{align*}
		\mathscr{L}(\cut(\phi))=0,
	\end{align*}
	where $\mathscr{L}$ denotes the Lebesgue measure on $\T^n$.	
\end{The}

\begin{proof}
	By Theorem \ref{thm:Reg_Cut}, we only need to prove $\mathscr{L}(\cut(\phi,H))=0$ for some specific Tonelli Hamiltonian $H$ on $\T^n$. We choose a special Hamiltonian $H(x,p)=\frac{1}{2} |p|^2$, $x\in\T^n$, $p\in\R^n$. In this case, the associated fundamental solution is given by  
	\[
	A_t(x,y) = \frac{d^2(x,y)}{2t}, \quad t>0,\ x,y\in\T^n.
	\]
	According to Proposition \ref{pro:uni lip semi}, there exist $C_1>0$ and $C_2>0$ such that the function $T^-_t\phi:\T^n\to\R$ is $C_1$-Lipschitz and $C_2$-semiconcave for all $t\geqslant0$.
	

	For $t>0$, consider the map $\Gamma_t:E(\phi,H,t)\to\T^n$ (see Definition \ref{defn:tau_phi} for $E(\phi,H,t)$) defined by
	\[
	x\mapsto\Gamma_t(x)=\bigg\{ y\in\T^n : T^-_t\phi(y) = \phi(x) + \frac{d^2(x,y)}{2t} \bigg\}.
	\]
	Notice that the Hamiltonian system in this case is given by  
	\[
	\begin{cases}
		\dot{x} = p,\\
		\dot{p} = 0.
	\end{cases}
	\]
	Thus, for $x\in E(\phi,H,t)$, by Proposition~\ref{pro:cali} (2) and (4), $D\phi(x)$ exists and we have
	\begin{equation}\label{eq:cut0pf1}
		\Gamma_t(x) = x + t D\phi(x),
	\end{equation}
	and
	\begin{equation}\label{eq:cut0pf2}
		D\phi(x) \in D^+T^-_t\phi(\Gamma_t(x)).
	\end{equation}
	Since the infimum in the definition of $T^-_t\phi$ can always be attained, the map $\Gamma_t$ is surjective, that is,
	\begin{equation}\label{eq:cut0pf3}
		\Gamma_t(E(\phi,H,t)) = \T^n.
	\end{equation}
	Now, for any ball $B(x_0, r)\subset\T^n$ with radius $0<r<\frac{1}{100}$ and $0<t<\min\{r/C_1, 1/C_2\}$, we obtain from \eqref{eq:cut0pf1} and \eqref{eq:cut0pf3} that
	\begin{equation}\label{eq:cut0pf4}
		B(x_0, r - tC_1) \subset \Gamma_t(E(\phi,H,t) \cap B(x_0, r)).
	\end{equation}
	Moreover, for any $x_1, x_2 \in E(\phi,H,t) \cap B(x_0, r)$, by \eqref{eq:cut0pf1}, \eqref{eq:cut0pf2}, the $C_2$-semiconcavity of $T^-_t\phi$ (see Proposition~\ref{pro:uni lip semi}) and \cite[Proposition 3.3.10]{Cannarsa_Sinestrari_book}, we compute
	\begin{align*}
		|x_1 - x_2| \cdot |\Gamma_t(x_1) - \Gamma_t(x_2)|
		&\geqslant \langle x_1 - x_2, \Gamma_t(x_1) - \Gamma_t(x_2) \rangle\\
		&= \langle (\Gamma_t(x_1) - tD\phi(x_1)) - (\Gamma_t(x_2) - tD\phi(x_2)), \Gamma_t(x_1) - \Gamma_t(x_2) \rangle\\
		&= |\Gamma_t(x_1) - \Gamma_t(x_2)|^2 - t \langle D\phi(x_1) - D\phi(x_2), \Gamma_t(x_1) - \Gamma_t(x_2) \rangle\\
		&\geqslant (1 - tC_2) |\Gamma_t(x_1) - \Gamma_t(x_2)|^2.
	\end{align*}
	This leads to the estimate
	\begin{equation}\label{eq:cut0pf5}
		|\Gamma_t(x_1) - \Gamma_t(x_2)| \leqslant \frac{1}{1 - tC_2} |x_1 - x_2|.
	\end{equation}
	Thus, from \eqref{eq:cut0pf4}, \eqref{eq:cut0pf5} and the relation $E(\phi,H,t)\subset\reg(\phi,H)$, we obtain the measure comparison
	\begin{align*}
		&\mathscr{L}(B(x_0, r - tC_1))
		\leqslant \mathscr{L}(\Gamma_t(E(\phi,H,t) \cap B(x_0, r)))
		\leqslant \bigg(\frac{1}{1 - tC_2} \bigg)^n \mathscr{L}(E(\phi,H,t) \cap B(x_0, r))\\
		\leqslant &\bigg(\frac{1}{1 - tC_2} \bigg)^n \mathscr{L}(\reg(\phi,H) \cap B(x_0, r))
		\leqslant \bigg(\frac{1}{1 - tC_2} \bigg)^n \mathscr{L}(B(x_0, r)).
	\end{align*}
	Sending $t \to 0^+$, we conclude that  
	\[
	\mathscr{L}(\reg(\phi,H) \cap B(x_0, r)) = \mathscr{L}(B(x_0, r)).
	\]
	Since the choice of $B(x_0, r)$ was arbitrary, it follows that  
	\[
	\mathscr{L}(\reg(\phi,H)) = \mathscr{L}(\T^n).
	\]
	This implies, by Theorem \ref{thm:Reg_Cut}, $\mathscr{L}(\cut(\phi))=\mathscr{L}(\cut(\phi,H))= 0$.
\end{proof}

\begin{Pro}[\protect{\cite[Lemma 4.1]{Azagra2023}}]\label{lem:level}
	Let \( \Omega \subset \R^n \) be an open convex subset, $\phi\in\text{\rm SCL}(\Omega)$ and $\psi\in C^{1,1}(\Omega)$ satisfy $\psi\leqslant\phi$. Define $A=\{x\in\Omega:\psi(x)=\phi(x)\}$. Then we have $\mathscr{L}(A)=\mathscr{L}(A\cap\text{\rm Alex}(\phi))$.
\end{Pro}

\begin{The}[Alexandrov's theorem]\label{thm:Alex}
	Let $\phi\in\text{\rm SCL}\,(\T^n)$. Then we have
\begin{align*}
	\mathscr{L}(\text{\rm Alex}\,(\phi))=\mathscr{L}(\T^n).
\end{align*}
\end{The}

\begin{proof}
	By \cite[Lemma 3.1, Proposition 3.2]{Cannarsa_Cheng_Hong2025}, we have that
	\begin{enumerate}[i)]
		\item $T^+_t \circ T^-_t \phi \leqslant \phi$ for all $t>0$, and $T^+_t \circ T^-_t \phi(x) = \phi(x)$ if and only if $x \in E(\phi,H,t)$.
		\item there exists $\tau_0 > 0$ such that $T^+_t \circ T^-_t \phi$ is of class $C^{1,1}$ for all $t \in (0, \tau_0]$.
	\end{enumerate}
	Then, from Proposition \ref{lem:level}, we conclude that
	\[
	\mathscr{L}(E(\phi,H,t)) = \mathscr{L}(E(\phi,H,t) \cap \text{Alex}(\phi)).
	\]
	Letting $t \to 0^+$ on both sides of the equality above, combined with Theorem \ref{thm:cut0}, we obtain
	\[
	\mathscr{L}(\T^n) = \mathscr{L}(\reg(\phi)) = \mathscr{L}(\reg(\phi) \cap \text{Alex}(\phi)) \leqslant \mathscr{L}(\text{Alex}(\phi)) \leqslant \mathscr{L}(\T^n),
	\]
	which implies $\mathscr{L}(\text{\rm Alex}\,(\phi))=\mathscr{L}(\T^n)$.
\end{proof}

\subsection{Criterion for closedness of the cut locus}

For any weak KAM solution $\phi$ of \eqref{eq:HJs} on $\T^n$, we recall the following results on regularity of $\phi$ (keeping in mind that $\reg(\phi,H) = \reg(\phi)$ and $\cut(\phi,H) = \cut(\phi)$, by Theorem \ref{thm:Reg_Cut}).
\begin{enumerate}
	\item $\sing(\phi)\subset \cut(\phi)\subset\overline{\sing(\phi)}$.
	\item $\text{supp}_{C^{1,1}}(\phi)=\T^n\setminus\overline{\sing(\phi)}$, where $\text{supp}_{C^{1,1}}(\phi)$ is the $C^{1,1}$ support of $\phi$ and consists of the points $x\in\T^n$ for which there exists a neighborhood $V$ of $x$ such that $\phi$ is of class $C^{1,1}$ in $V$.
	\item The cut time function $\tau_{\phi,H}$ is upper-semicontinuous (see Proposition \ref{pro:tau-usc} or \cite{Cannarsa_Cheng_Fathi2021}).
	\item $R_\phi$ is a lower semicontinuous function on $\T^n$ (Proposition \ref{lem:R_phi}).
\end{enumerate}

Now we give the following characterization of the closedness of $\cut(\phi)$.

\begin{The}\label{thm:close_cut}
Suppose $H$ is a Tonelli Hamiltonian on $\T^n$ and $\phi$ is a weak KAM solution of \eqref{eq:HJs}. Then the following statements are equivalent.
\begin{enumerate}[\rm (a)]
	\item $\cut(\phi)$ is closed.
	\item $\cut(\phi)=\overline{\sing(\phi)}$.
	\item For any $(\phi,H)$-calibrated curve $\gamma:(-\infty,0]\to\T^n$ we have $\gamma(t)\in \text{\rm supp}_{C^{1,1}}(\phi)$ for all $t<0$.
	\item $\tau_{\phi,H}$ is continuous (in the extended interval $[0,+\infty]$).
	\item For any $x\in\reg(\phi)$, the function $R_\phi$ is locally bounded above at $x$.
\end{enumerate}
\end{The}

\begin{proof}
(a) $\Rightarrow$ (b): We conclude from $\sing(\phi)\subset \cut(\phi)\subset\overline{\sing(\phi)}$ that $\overline{\cut(\phi)}=\overline{\sing(\phi)}$. Since $\cut(\phi)$ is closed, we have $\cut(\phi)=\overline{\cut(\phi)}$. Thus, $\cut(\phi)=\overline{\sing(\phi)}$.

(b) $\Rightarrow$ (c): If $\gamma:(-\infty,0]\to\T^n$ is a $(\phi,H)$-calibrated curve, then $\gamma(t)\in\T^n\setminus\cut(\phi)$ for all $t<0$. It follows that $\gamma(t)\in\text{supp}_{C^{1,1}}(\phi)$, since $\cut(\phi)=\overline{\sing(\phi)}$ and $\T^n\setminus\overline{\sing(\phi)}=\text{supp}_{C^{1,1}}(\phi)$.

(c) $\Rightarrow$ (d): Since $\tau_{\phi,H}$ is upper semicontinuous by Proposition \ref{pro:tau-usc}, it is sufficient to show $\tau_{\phi,H}$ is lower semicontinuous. We prove it by contradiction. Suppose that for some $x\in\T^n$
\begin{align*}
	\liminf_{y\to x}\tau_{\phi,H}(y)=\bar{t}<\tau_{\phi,H}(x).
\end{align*}
Then there exists a sequence $y_i\to x$, as $i\to\infty$, such that $\lim_{i\to\infty}\tau_{\phi,H}(y_i)=\bar{t}$. 

Fix $\delta>0$ such that $\bar{t}+\delta<\tau_{\phi,H}(x)$, and let $\gamma:[0,\bar{t}+\delta]\to\T^n$ be the $(\phi,H)$-calibrated curve starting from $\gamma(0)=x$. If, by taking a subsequence, $\tau_{\phi,H}(y_i)=0$ for all $i\in\N$, we have $y_i\in\cut(\phi)$ for all $i\in\N$. This implies $x\in\overline{\cut(\phi)}=\overline{\sing(\phi)}$, which contradicts (c) since $x=\gamma(0)\in\text{supp}_{C^{1,1}}(\phi)=\T^n\setminus\overline{\sing(\phi)}$. Thus, by taking a subsequence, we have $0<\tau_{\phi,H}(y_i)<+\infty$ for all $i\in\N$. Now for any $i\in\N$, by Lemma \ref{lem:uniq cali}, there exists a unique $(\phi,H)$-calibrated curve $\gamma_i:[0,\tau_{\phi,H}(y_i)]\to\T^n$ starting from $\gamma_i(0)=y_i$, and we have
\begin{equation}\label{eq:cri pf}
	\gamma(t)=\pi\circ\Phi^t_H(x,D\phi(x)),\quad \forall t\in[0,\bar{t}+\delta],\qquad \gamma_i(t)=\pi\circ\Phi^t_H(y_i,D\phi(y_i)),\quad \forall t\in[0,\tau_{\phi,H}(y_i)],
\end{equation}
where $\pi:T^*\T^n\to\T^n$ is the canonical projection and $\Phi^t_H$ is the Hamiltonian flow for $H$. Due to the upper semicontinuity of the set-valued map $x\rightrightarrows D^+\phi(x)$, we have $\lim_{i\to\infty}D\phi(y_i)=D\phi(x)$. Together with \eqref{eq:cri pf} and the relation $\lim_{i\to\infty}\tau_{\phi,H}(y_i)=\bar{t}$,
\begin{equation}\label{eq:cri pf 2}
	\lim_{i\to\infty}\gamma_i(\tau_{\phi,H}(y_i))=\lim_{i\to\infty}\pi\circ\Phi^{\tau_{\phi,H}(y_i)}_H(y_i,D\phi(y_i))=\pi\circ\Phi^{\bar{t}}_H(x,D\phi(x))=\gamma(\bar{t}).
\end{equation}
If $\gamma_i(\tau_{\phi,H}(y_i))$ were regular, it would admit a $(\phi,H)$-calibrated continuation for a positive time interval. Concatenation would then extend $\gamma_i$ beyond $\tau_{\phi,H}(y_i)$, contradicting the maximality in Definition \ref{defn:tau_phi}. Hence $\gamma_i(\tau_{\phi,H}(y_i))\in\cut(\phi)$ for all $i\in\N$. Consequently, we obtain that $\gamma(\bar{t})\in\overline{\cut(\phi)}=\overline{\sing(\phi)}$ by \eqref{eq:cri pf 2}. This contradicts (c) which implies $\gamma(\bar{t})\in\text{supp}_{C^{1,1}}(\phi)=\T^n\setminus\overline{\sing(\phi)}$. Hence we proved
\begin{align*}
	\liminf_{y\to x}\tau_{\phi,H}(y)\geqslant\tau_{\phi,H}(x),\quad \forall x\in\T^n,
\end{align*}
i.e., $\tau_{\phi,H}$ is lower semicontinuous.

(d) $\Rightarrow$ (e): Let $x\in\reg(\phi)$ and choose a finite number $a\in(0,\tau_{\phi,H}(x))$. By continuity of $\tau_{\phi,H}$ in $[0,+\infty]$, there is a neighborhood $U$ of $x$ on which $\tau_{\phi,H}(y)\geqslant a$. Theorem~\ref{thm:1} (1) then gives
\[
R_\phi(y)\leqslant\max\{C_{1,1}/a,C_{1,2}\},\qquad y\in U.
\]
This covers finite and infinite cut times simultaneously.

(e) $\Rightarrow$ (a): By statement (e), for any $x\in\reg(\phi)$ there exists an open neighborhood $U$ of $x$ and $\lambda>0$, such that $R_\phi(y)\leqslant\lambda$ for all $y\in U$. It follows that $U\subset\reg(\phi)$, which implies $\reg(\phi)$ is open. Thus, $\cut(\phi)=\T^n\setminus\reg(\phi)$ is closed.
\end{proof}

\subsection{Propagation of cut points and Alexandrov points}

In this section, we apply the results obtained above to two problems for the evolutionary Hamilton--Jacobi equation \eqref{eq:HJe}, namely the global propagation of cut points along generalized characteristics and the propagation of Alexandrov points along calibrated curves, together with the evolution of the second derivative according to the matrix Riccati equation.

Let $H$ be a Tonelli Hamiltonian on $\T^n$, $\phi\in\SCL(\T^n)$, and $u$ be the viscosity solution to the Cauchy problem of evolutionary Hamilton--Jacobi equation \eqref{eq:HJephi pre}.
We call an absolutely continuous curve $\gamma:I\to\T^n$ a $(u,H)$-generalized characteristic if it satisfies the differential inclusion
\begin{equation}\label{eq:GC}\tag{GC}
		\dot{\gamma}(s)\in\text{\rm co}\,H_p(\gamma(s),D^+_xu(s,\gamma(s))),\qquad a.e.\ s\in I,
\end{equation}
where $\text{\rm co}$ denotes the convex hull. It is well known that if $\gamma:[0,T]\to\T^n$ is a $(u,H)$-calibrated curve, $D_x u(s,\gamma(s))$ exists for all $s\in[0,T)$, and it satisfies the classical characteristic equation
\begin{equation}\tag{CH}
	\dot{\gamma}(s)=H_p(\gamma(s),D_xu(s,\gamma(s))),\qquad \forall s\in[0,T).
\end{equation}
Obviously, a $(u,H)$-calibrated curve is a $(u,H)$-generalized characteristic. The following Theorem \ref{distance estimate} is a key tool for the propagation results of both cut points and Alexandrov points.

\begin{The}\label{distance estimate}
	Let $H$ be a Tonelli Hamiltonian on $\T^n$, $\phi\in\SCL(\T^n)$, and $u$ be the viscosity solution of \eqref{eq:HJephi pre}. Then there exist $C_1>0$ and $C_2>0$ (depending on the local $C^2$ modulus and the Tonelli modulus of $H$ and the Lipschitz and semiconcavity constants of $\phi$) such that for any $(u,H)$-calibrated curve $\gamma:[0,T]\to\T^n$ and $(u,H)$-generalized characteristic $\xi:[0,t]\to\T^n$ satisfying \eqref{eq:GC} with $0\leqslant t<T$,
	\begin{equation}\label{eq:curve no conju}
		e^{-C_1 T}(T-t)^{C_2}\cdot d(\xi(0),\gamma(0))\leqslant d(\xi(t),\gamma(t))\leqslant e^{C_1 T}(T-t)^{-C_2} \cdot d(\xi(0),\gamma(0)).
	\end{equation}
\end{The}

\begin{proof}
	Notice that for any $s\in[0,t]$, we have $\tau_{u(s,\cdot),H}(\gamma(s))\geqslant T-s$. Invoking Proposition \ref{pro:uni lip semi} and Theorem \ref{thm:1} (1), there exist positive constants $C_{1,1}$, $C_{1,2}$ such that
	\begin{align*}
		R_{u(s,\cdot)}(\gamma(s))\leqslant\max\big\{\frac{C_{1,1}}{\tau_{u(s,\cdot),H}(\gamma(s))},C_{1,2}\big\}\leqslant\max\big\{\frac{C_{1,1}}{T-s},C_{1,2}\big\},\quad \forall s\in[0,t].
	\end{align*}
	This implies
	\begin{align*}
		|p-D_x u(s,\gamma(s))|\leqslant\max\big\{\frac{C_{1,1}}{T-s},C_{1,2}\big\}\cdot d(x,\gamma(s)),\quad \forall s\in[0,t],\ x\in\T^n,\ p\in D^+_x u(s,x).
	\end{align*}
	Thus, we have
	\begin{equation}\label{eq:sep pf 1}
		\begin{split}
			\left|\frac{d}{ds}d(\xi(s),\gamma(s))\right|\leqslant\ &|\dot{\xi}(s)-\dot{\gamma}(s)|
			=\ \bigg|\int_{D^+_xu(s,\xi(s))}H_p(\xi(s),p)\ d\nu_s-H_p(\gamma(s),D_xu(s,\gamma(s))\bigg|\\
			\leqslant\ &\,C_0\cdot d(\xi(s),\gamma(s))+C_0\int_{D^+_xu(s,\xi(s))}|p-D_xu(s,\gamma(s))|\ d\nu_s\\
			\leqslant\ &\,C_0\left(1+C_{1,2}+\frac{C_{1,1}}{T-s}\right)\cdot d(\xi(s),\gamma(s)),\quad a.e.\, s\in[0,t],
		\end{split}
	\end{equation}
	where $\{\nu_s\}$ is a family of probability measures supported on $D^+_xu(s,\xi(s))$ and
	\begin{align*}
		C_0=1+\sup_{x\in\T^n,\ |p|\leqslant P_*}\big\|D_{(x,p)}H_p(x,p)\big\|,
	\end{align*}
	with $P_*$ a uniform bound for the spatial supergradients of $u$, provided by Proposition \ref{pro:uni lip semi}. In the second inequality above we used the mean-value estimate
	\[
	|H_p(x,p)-H_p(y,q)|\leqslant C_0\big(d(x,y)+|p-q|\big),
	\]
	which holds for all $x,y\in\T^n$ and $|p|,|q|\leqslant P_*$. The resulting differential inequality is two-sided, so Gronwall's inequality applied to \eqref{eq:sep pf 1} and in reverse directly yields both sides of the separation estimate. Employing Gronwall's inequality to \eqref{eq:sep pf 1} in positive direction, we have
	\begin{align*}
		d(\xi(t),\gamma(t))&\leqslant e^{C_0(1+C_{1,2})t}(\frac{T}{T-t})^{C_0 C_{1,1}}\cdot d(\xi(0),\gamma(0))\\
		&\leqslant e^{C_0(1+C_{1,2})T}T^{C_0 C_{1,1}}(T-t)^{-C_0 C_{1,1}}\cdot d(\xi(0),\gamma(0))\\
		&\leqslant e^{C_0(1+C_{1,1}+C_{1,2})T}(T-t)^{-C_0 C_{1,1}}\cdot d(\xi(0),\gamma(0)).
	\end{align*}
	Similarly, employing Gronwall's inequality to \eqref{eq:sep pf 1} in negative direction, we have
	\begin{align*}
		d(\xi(0),\gamma(0))\leqslant e^{C_0(1+C_{1,1}+C_{1,2})T}(T-t)^{-C_0 C_{1,1}}\cdot d(\xi(t),\gamma(t)).
	\end{align*}
	Taking $C_1=C_0(1+C_{1,1}+C_{1,2})$ and $C_2=C_0 C_{1,1}$ we can obtain \eqref{eq:curve no conju}.
\end{proof}

\begin{Rem}
	For evolutionary equation \eqref{eq:HJephi pre} with $C^2$ initial data $\phi$, it is well known that \emph{no conjugate points occur before the cut locus}, that is, the differential of the endpoint map defined by Hamiltonian flow is non-degenerate along calibrated curves before the cut locus (\cite[Theorem 3.6]{Cannarsa_Sinestrari_book}). The separation estimate \eqref{eq:curve no conju} in Theorem \ref{distance estimate} expresses this fact for more general initial data $\phi$, which is only semiconcave with linear modulus, from the trajectory point of view.
\end{Rem}

\subsubsection{Global propagation of cut points}

\begin{sloppypar}
Now we consider the evolutionary Hamilton--Jacobi equation \eqref{eq:HJephi pre}. If $u:(0,\infty)\times\T^n\to\R$ is a viscosity solution of \eqref{eq:HJephi pre}, then $u\in\text{\rm SCL}_{\rm loc}\,((0,\infty)\times\T^n)$ and $u(t,\cdot)\in\text{\rm SCL}\,(\T^n)$ for all $t>0$.
\end{sloppypar}

\begin{Pro}\label{pro:Reg_Cut2}
Let $H$ be a Tonelli Hamiltonian on $\T^n$ and $u$ be a viscosity solution of \eqref{eq:HJe}. Then there holds
\begin{align*}
	\reg(u)=\bigcup_{t>0}\{t\}\times\reg(u(t,\cdot))=\reg(u,H)\cap\big((0,\infty)\times\T^n\big),\\
	\cut(u)=\bigcup_{t>0}\{t\}\times\cut(u(t,\cdot))=\cut(u,H)\cap\big((0,\infty)\times\T^n\big).
\end{align*}
In other words,
\begin{align*}
	\partial_P u(t,x)\neq\varnothing \Longleftrightarrow \partial_P u(t,\cdot)(x)\neq\varnothing,\qquad \forall (t,x)\in(0,+\infty)\times\T^n.
\end{align*}
\end{Pro}

\begin{proof}
Let $(t,x)\in\reg(u)$. In view of Lemma \ref{lem:smooth touching}, $u$ is differentiable at $(t,x)$, and there exists an open neighborhood $\tilde{U}=(t-\delta,t+\delta)\times U$ of $(t,x)$ and $w\in C^{\infty}(\tilde{U})$ touching $u$ from below at $(t,x)$. It follows that $w(t,\cdot)\in C^{\infty}(U)$ touches $u(t,\cdot)$ from below at $x$. Thus $x\in\reg(u(t,\cdot))$ and
\begin{align*}
	\reg(u)\subset\bigcup_{t>0}\{t\}\times\reg(u(t,\cdot)).
\end{align*}

Now we turn to the opposite direction. Let $t>0$ and $x\in\reg(u(t,\cdot))$. By Lemma \ref{lem:smooth touching} and a standard cutoff argument, there exists $\psi\in C^{\infty}(\T^n)$ touching $u(t,\cdot)$ from below at $x$. Define
\begin{align*}
	v(s,y)=
	\begin{cases}
		T^-_{s-t}\psi(y),&s\geqslant t, y\in\T^n;\\
		T^+_{t-s}\psi(y),&s<t, y\in\T^n.
	\end{cases}
\end{align*}
It is clear that there exists $\delta>0$ such that $v$ is of class $C^{1,1}$ and it is a viscosity solution of \eqref{eq:HJe} on $(t-\delta,t+\delta)\times\T^n$ (see, for instance, \cite[Proposition 3.2]{Cannarsa_Cheng_Hong2025}). Therefore, for $t\leqslant s<t+\delta$ and $y\in\T^n$, by the order-preserving property of semigroups $\{T^{\pm}_t\}_{t\geqslant0}$, we have
\begin{align*}
	v(s,y)=T^-_{s-t}\psi(y)\leqslant (T^-_{s-t}u(t,\cdot))(y)=u(s,y).
\end{align*}
Similarly, by \cite[Corollary 3.15]{Cannarsa_Cheng_Hong2025}, for $t-\delta< s<t$ and $y\in\T^n$ we have
\begin{align*}
	v(s,y)=T^+_{t-s}\psi(y)\leqslant (T^+_{t-s}u(t,\cdot))(y)\leqslant u(s,y).
\end{align*}
It follows that $v$ touches $u$ from below at $(t,x)$ and this implies $(t,x)\in\reg(u)$. Thus
\begin{align*}
	\bigcup_{t>0}\{t\}\times\reg(u(t,\cdot))\subset\reg(u).
\end{align*}

Now we have
\begin{align*}
	\reg(u)=\bigcup_{t>0}\{t\}\times\reg(u(t,\cdot)),
\end{align*}
and consequently,
\begin{align*}
	\cut(u)=&\,\{(0,\infty)\times\T^n\}\setminus\reg(u)
	=\,\bigg(\bigcup_{t>0}\{t\}\times\T^n\bigg)\setminus\bigg(\bigcup_{t>0}\{t\}\times\reg(u(t,\cdot))\bigg)\\
	=&\,\bigcup_{t>0}\{t\}\times\big(\T^n\setminus\reg(u(t,\cdot))\big)
	=\,\bigcup_{t>0}\{t\}\times\cut(u(t,\cdot)).
\end{align*} 
Finally, relation \eqref{eq:cut u rel} and Proposition \ref{pro:prox} lead to the remaining relations and equivalence.
\end{proof}

The following Corollary \ref{cor:evcut0} is a direct consequence of Theorem \ref{thm:cut0}, Proposition \ref{pro:Reg_Cut2} and Fubini's theorem.

\begin{Cor}\label{cor:evcut0}
	Let $H$ be a Tonelli Hamiltonian on $\T^n$ and $u$ be a viscosity solution of \eqref{eq:HJe}. Then we have $\mathscr{L}(\cut(u))=0$.
\end{Cor}

The following Theorem \ref{thm:pro_cut} is the global propagation of cut points along generalized characteristics for the evolutionary Hamilton--Jacobi equation.

\begin{The}\label{thm:pro_cut}
Let $H$ be a Tonelli Hamiltonian on $\T^n$ and $u$ be a viscosity solution of \eqref{eq:HJe}.
\begin{enumerate}[\rm (1)]
	\item If $\bar{\gamma}:(0,T]\to\T^n$ is a $(u,H)$-calibrated curve, then for any $t\in(0,T)$, $\bar{\gamma}$ is the unique solution of \eqref{eq:GC} with $\gamma(t)=\bar{\gamma}(t)$ on $(0,T]$.
	\item If $\gamma:[t_0,\infty)\to\T^n$ is a solution of \eqref{eq:GC} with $(t_0,\gamma(t_0))\in\cut(u)$, then $(s,\gamma(s))\in\cut(u)$ for all $s\geqslant t_0$.
\end{enumerate}
\end{The}

\begin{proof}
Let $\gamma:(0,T]\to\T^n$ be a solution of \eqref{eq:GC} with $\gamma(t)=\bar{\gamma}(t)$. By Theorem \ref{distance estimate}, for any $s\in(0,T)$, there exists $C_s>0$ such that
\begin{align*}
	d(\gamma(s),\bar{\gamma}(s))\leqslant C_s\cdot d(\gamma(t),\bar{\gamma}(t))=0,
\end{align*}
that is, $\gamma(s)=\bar{\gamma}(s)$. Letting $s\to T$, we have $\gamma(T)=\bar{\gamma}(T)$. This implies the required uniqueness.

We prove (2) by contradiction. Suppose there exists $t>t_0$ such that $(t,\gamma(t))\in\reg(u)$. By Proposition \ref{pro:Reg_Cut2} and Theorem \ref{thm:Reg_Cut}, we observe that $\gamma(t)\in\reg(u(t,\cdot))=\reg(u(t,\cdot),H)$. Thus, there exists $\delta>0$ and a $(u,H)$-calibrated curve $\bar{\gamma}:(0,t+\delta]\to\T^n$ such that $\bar{\gamma}(t)=\gamma(t)$. From (1) we have $\gamma(s)=\bar{\gamma}(s)$ for all $s\in[t_0,t+\delta]$. This implies $\gamma(t_0)=\bar{\gamma}(t_0)\in\reg(u(t_0,\cdot),H)=\reg(u(t_0,\cdot))$, i.e., $(t_0,\gamma(t_0))\in\reg(u)$. This leads to a contradiction with $(t_0,\gamma(t_0))\in\cut(u)$. Thus we have $(s,\gamma(s))\in\cut(u)$ for all $s\geqslant t_0$.
\end{proof}

\begin{Rem}
For weak KAM solutions of the stationary equation \eqref{eq:HJs}, the statement of Theorem~\ref{thm:pro_cut} was obtained in \cite[Theorem 5.7]{CCHW2024}.
\end{Rem}

\subsubsection{Propagation of Alexandrov points}

In this subsection, we investigate the propagation of Alexandrov points of the viscosity solution to \eqref{eq:HJe}, together with the evolution of the second derivative governed by the matrix Riccati equation.

\begin{The}\label{the:propagate alex}
	Suppose $H$ is a Tonelli Hamiltonian on $\T^n$, $\phi\in \SCL(\T^n)$, and $u$ is the viscosity solution of \eqref{eq:HJephi pre}. Let $\gamma: [0,T] \to \T^n$ be a $(u,H)$-calibrated curve with dual arc $p(t) = L_v(\gamma(t),\dot{\gamma}(t))$ for $t \in [0,T]$. The following statements are true:
	
	\begin{enumerate}
		\item 	If $\gamma(t_0)\in\operatorname{Alex}(u(t_0,\cdot))$ for some $t_0\in[0,T)$, then $\gamma(t)\in\operatorname{Alex}(u(t,\cdot))$ for all $t\in[t_0,T)$. Moreover, the map $t\mapsto Q(t)=D^2_xu(t,\gamma(t))$ is of class $C^1$ on $[t_0,T)$ and $Q(t)$ satisfies the following matrix Riccati equation
		\begin{equation}\label{eq:Ricc}
			\dot{Q}(t)+Q(t)\cdot H_{pp}(\gamma(t),p(t))\cdot Q(t)+Q(t)\cdot H_{px}(\gamma(t),p(t))+H_{xp}(\gamma(t),p(t))\cdot Q(t)+H_{xx}(\gamma(t),p(t))=0. \tag{Ricc}
		\end{equation}		
		\item  If $\gamma(t_0)\in\operatorname{Alex}(u(t_0,\cdot))\cap \operatorname{supp}_{C^{1,1}}(u(t_0,\cdot))$ for some $t_0\in[0,T)$, where $\text{supp}_{C^{1,1}}(u(t_0,\cdot))$ is the $C^{1,1}$ support of $u(t_0,\cdot)$ (see the definition before Theorem \ref{thm:close_cut}), then $\gamma(t)\in\operatorname{Alex}(u(t,\cdot))$ for all $t\in[0,T)$. Moreover, the map $t\mapsto Q(t)=D^2_xu(t,\gamma(t))$ is of class $C^1$ on $[0,T)$ and $Q(t)$ satisfies \eqref{eq:Ricc}.
	\end{enumerate}
\end{The}

\begin{proof}
	We first observe that the pair $(\gamma,p)$ satisfies the Hamiltonian equation
	\begin{equation}\label{eq:H ode 1}
		\begin{cases}
			\dot{\gamma}(s)=H_p(\gamma(s),p(s))\\
			\dot{p}(s)=-H_x(\gamma(s),p(s))
		\end{cases}
		\qquad s\in[0,T].
	\end{equation}

	Fix $t\in(t_0,T)$ and let $\tilde{\gamma}:[t_0,t]\to\T^n$ be an arbitrary $(u,H)$-calibrated curve, with $\tilde{p}(\cdot)$ the dual arc of $\tilde{\gamma}(\cdot)$. Then, $\tilde{p}(s)=D_xu(s,\tilde{\gamma}(s))$ for all $s\in[t_0,t)$ and $\tilde{p}(t)\in D^*_xu(t,\tilde{\gamma}(t))$. The pair $(\tilde{\gamma},\tilde{p})$ also satisfies Hamiltonian equation
	\begin{equation}\label{eq:H ode 2}
		\begin{cases}
			\dot{\tilde{\gamma}}(s)=H_p(\tilde{\gamma}(s),\tilde{p}(s))\\
			\dot{\tilde{p}}(s)=-H_x(\tilde{\gamma}(s),\tilde{p}(s))
		\end{cases}
		\qquad s\in[t_0,t].
	\end{equation}

	Without loss of generality we work on the universal covering space $\R^n$. Set
	\[
	\bar{\gamma}(s)=\tilde{\gamma}(s)-\gamma(s),\qquad \bar{p}(s)=\tilde{p}(s)-p(s),\qquad s\in[t_0,t].
	\]
	
	Since $\gamma(t_0)\in\operatorname{Alex}(u(t_0,\cdot))$, due to Proposition \ref{pro:Alex2} we have that
	\begin{equation}\label{eq:alex pf 1}
		\begin{split}
			\bar{p}(t_0)&=\tilde{p}(t_0)-p(t_0)=D_xu(t_0,\tilde{\gamma}(t_0))-D_xu(t_0,\gamma(t_0))\\
			&=Q(t_0)\cdot(\tilde{\gamma}(t_0)-\gamma(t_0))+o(|\tilde{\gamma}(t_0)-\gamma(t_0)|)=Q(t_0)\cdot\bar{\gamma}(t_0)+o(|\bar{\gamma}(t_0)|).
		\end{split}
	\end{equation}
	
	In view of the Lipschitz dependence of the data at time $t_0$ for the two Hamiltonian equations \eqref{eq:H ode 1} and \eqref{eq:H ode 2} and the relation \eqref{eq:alex pf 1} above, we conclude that there exists $C>0$ such that 
	\begin{equation}\label{eq:alex pf 01}
		\max	\{|\bar{\gamma}(s)|,|\bar{p}(s)|\}\leqslant C|\bar{\gamma}(t_0)|,\qquad \forall s\in[t_0,t],
	\end{equation}
	for sufficiently small $|\bar{\gamma}(t_0)|$. By Proposition \ref{pro:cali} (3), both $(\gamma,p)$ and $(\tilde{\gamma},\tilde{p})$ are uniformly bounded. Thus, combining \eqref{eq:H ode 1}, \eqref{eq:H ode 2} and \eqref{eq:alex pf 01} we obtain that
	\begin{equation}\label{eq:alex pf 2}
		\begin{pmatrix}
			\dot{\bar{\gamma}}(s)\\
			\dot{\bar{p}}(s)
		\end{pmatrix}
		=
		\begin{pmatrix}
			H_{px}(\gamma(s),p(s)) & H_{pp}(\gamma(s),p(s))\\
			-H_{xx}(\gamma(s),p(s)) & -H_{xp}(\gamma(s),p(s))
		\end{pmatrix}
		\cdot
		\begin{pmatrix}
			\bar{\gamma}(s)\\
			\bar{p}(s)
		\end{pmatrix}+o(|\bar\gamma(t_0)|),\quad s\in[t_0,t].
	\end{equation}
	
	Invoking Theorem \ref{distance estimate}, there exist constants $C_1,C_2>0$ depending on $(u,H)$ and $T$ such that
	\begin{equation}\label{distance estimate 2}
		\frac{(T-t)^{C_2}}{e^{C_1(T-t_0)}}|\bar{\gamma}(t_0)|\leqslant|\bar{\gamma}(t)|\leqslant \frac{e^{C_1(T-t_0)}}{(T-t)^{C_2}}|\bar{\gamma}(t_0)|.
	\end{equation}
	Hence the $o(|\bar\gamma(t_0)|)$ remainders in \eqref{eq:alex pf 1} and \eqref{eq:alex pf 2} are uniform over all $\tilde{\gamma}$ ending near $\gamma(t)$ and all $s\in[t_0,t]$. Indeed, by Proposition \ref{pro:cali} (3) and \eqref{eq:alex pf 01}, all the curves under consideration lie in a compact region of the phase space $T^*\T^n$, where the Taylor remainder of the Hamiltonian vector field is uniform. As a system of homogeneous linear ordinary differential equations with variable coefficients, equation
	\[
	\begin{pmatrix}
		\dot{\bar{\gamma}}(s)\\
		\dot{\bar{p}}(s)
	\end{pmatrix}
	=
	\begin{pmatrix}
		H_{px}(\gamma(s),p(s)) & H_{pp}(\gamma(s),p(s))\\
		-H_{xx}(\gamma(s),p(s)) & -H_{xp}(\gamma(s),p(s))
	\end{pmatrix}
	\cdot
	\begin{pmatrix}
		\bar{\gamma}(s)\\
		\bar{p}(s)
	\end{pmatrix},\quad s\in[0,T],
	\]
	admits $2n\times 2n$ state-transition matrix
	\[
	\Phi(t_2,t_1)=
	\begin{pmatrix}
		\Phi_1(t_2,t_1) & \Phi_2(t_2,t_1)\\
		\Phi_3(t_2,t_1) & \Phi_4(t_2,t_1)
	\end{pmatrix}
	\qquad 0\leqslant t_1\leqslant t_2\leqslant T,
	\]
	where each $\Phi_k(t_2,t_1)$ is an $n\times n$ block matrix, $k=1,2,3,4$. By \eqref{eq:alex pf 1} and \eqref{eq:alex pf 2} we have
	\begin{align*}
		\begin{pmatrix}
			\bar{\gamma}(t)\\
			\bar{p}(t)
		\end{pmatrix}
		&=\Phi(t,t_0)\cdot
		\begin{pmatrix}
			\bar{\gamma}(t_0)\\
			\bar{p}(t_0)
		\end{pmatrix}+o(|\bar\gamma(t_0)|)
		=\Phi(t,t_0)\cdot
		\begin{pmatrix}
			\bar{\gamma}(t_0)\\
			Q(t_0)\bar{\gamma}(t_0)+o(|\bar{\gamma}(t_0)|)
		\end{pmatrix}+o(|\bar\gamma(t_0)|)\\
		&=
		\begin{pmatrix}
			\Phi_1(t,t_0) & \Phi_2(t,t_0)\\
			\Phi_3(t,t_0) & \Phi_4(t,t_0)
		\end{pmatrix}
		\cdot
		\begin{pmatrix}
			\bar{\gamma}(t_0)\\
			Q(t_0)\bar{\gamma}(t_0)
		\end{pmatrix}
		+o(|\bar{\gamma}(t_0)|).
	\end{align*}
	
	That is
	\begin{equation}\label{eq: bar ode}
		\begin{cases}
			\bar{\gamma}(t)=(\Phi_1(t,t_0)+\Phi_2(t,t_0)Q(t_0))\cdot\bar{\gamma}(t_0)+o(|\bar{\gamma}(t_0)|)\\
			\bar{p}(t)=(\Phi_3(t,t_0)+\Phi_4(t,t_0)Q(t_0))\cdot\bar{\gamma}(t_0)+o(|\bar{\gamma}(t_0)|).
		\end{cases}
	\end{equation}
	
	We claim that the matrix $\Phi_1(t,t_0)+\Phi_2(t,t_0)Q(t_0)$ is invertible. Otherwise
	\[
	V=(\Phi_1(t,t_0)+\Phi_2(t,t_0)Q(t_0))(\R^n)
	\]
	is a linear subspace of $\R^n$ with $\dim V\leqslant n-1$. Choose a unit vector $\theta\in\R^n$ orthogonal to $V$. Observe that for any $\tau>0$ there exists a $(u,H)$-calibrated curve $\tilde{\gamma}_\tau:[0,t]\to\T^n$ such that $\tilde{\gamma}_\tau(t)=\gamma(t)+\tau\theta$. Then $\bar{\gamma}_\tau(t)=\tilde{\gamma}_\tau(t)-\gamma(t)=\tau\theta$ is orthogonal to $V$, and $|\bar{\gamma}_\tau(t)|=\tau$. By \eqref{eq: bar ode} and \eqref{distance estimate 2}, we have
	\[
	\limsup_{\tau\to 0^+}\frac{|\bar{\gamma}_\tau(t)|}{|\bar{\gamma}_\tau(t_0)|}\leqslant\limsup_{\tau\to 0^+}\frac{|\bar{\gamma}_\tau(t)-(\Phi_1(t,t_0)+\Phi_2(t,t_0)Q(t_0))\cdot\bar{\gamma}_\tau(t_0)|}{|\bar{\gamma}_\tau(t_0)|}=0.
	\]
	This contradicts the relation in \eqref{distance estimate 2}.
	
	Now, since $\Phi_1(t,t_0)+\Phi_2(t,t_0)Q(t_0)$ is invertible, by \eqref{eq: bar ode} and \eqref{distance estimate 2} we obtain
	\[
	\bar{p}(t)=M(t,t_0)\cdot\bar{\gamma}(t)+o(|\bar{\gamma}(t)|),
	\]
	where we set
	\[
	M(t,t_0)=(\Phi_3(t,t_0)+\Phi_4(t,t_0)Q(t_0))\cdot(\Phi_1(t,t_0)+\Phi_2(t,t_0)Q(t_0))^{-1}.
	\]
	In other words,
	\begin{equation}\label{eq:alex pf 3}
		\lim_{\tilde{\gamma}(t)\to\gamma(t)}\frac{|\tilde{p}(t)-D_xu(t,\gamma(t))-M(t,t_0)\cdot(\tilde{\gamma}(t)-\gamma(t))|}{|\tilde{\gamma}(t)-\gamma(t)|}
		=\lim_{|\bar{\gamma}(t)|\to 0}\frac{|\bar{p}(t)-M(t,t_0)\cdot\bar{\gamma}(t)|}{|\bar{\gamma}(t)|}=0.
	\end{equation}
	
	Recall that $u$ is a viscosity solution of \eqref{eq:HJe}. For any $y\in\T^n$ and $p\in D^*_xu(t,y)$ there exists a $(u,H)$-calibrated curve $\tilde{\gamma}:[0,t]\to\T^n$ with $\tilde{\gamma}(t)=y$ and $\tilde{p}(t)=p$. Therefore, using the relation $D^+_xu(t,y)=\operatorname{co}D^*_xu(t,y)$ and \eqref{eq:alex pf 3} we have that
	\begin{align*}
		&\lim_{y\to\gamma(t)}\sup_{p\in D^+_xu(t,y)}\frac{|p-D_xu(t,\gamma(t))-M(t,t_0)\cdot(y-\gamma(t))|}{|y-\gamma(t)|}\\
		=&\lim_{y\to\gamma(t)}\sup_{p\in D^*_xu(t,y)}\frac{|p-D_xu(t,\gamma(t))-M(t,t_0)\cdot(y-\gamma(t))|}{|y-\gamma(t)|}\\
		=&\,0.
	\end{align*}
	
	This implies that $\gamma(t)\in\operatorname{Alex}(u(t,\cdot))$ for all $t\in[t_0,T)$ by Proposition \ref{pro:Alex2} and
	\begin{equation}\label{eq:alex pf 4}
		Q(t)=D^2_xu(t,\gamma(t))=M(t,t_0).
	\end{equation}
	
	Notice that the map $t\mapsto Q(t)$, $t\in[t_0,T)$ is of class $C^1$ since $t\mapsto\Phi(t,t_0)$ is of class $C^1$ and $Q(t)$ satisfies \eqref{eq:alex pf 4}. Moreover, for any $t_0\leqslant t\leqslant s<T$, \eqref{eq:alex pf 4} leads to
	\[
	Q(s)=(\Phi_3(s,t)+\Phi_4(s,t)Q(t))\cdot(\Phi_1(s,t)+\Phi_2(s,t)Q(t))^{-1}.
	\]
	
	Set $U(s)=\Phi_3(s,t)+\Phi_4(s,t)Q(t)$ and $V(s)=\Phi_1(s,t)+\Phi_2(s,t)Q(t)$. 	
	Using the matrix quotient rule, we obtain
	\begin{align*}
		\dot{Q}(s) &= \dot{U}(s)V(s)^{-1} - U(s)V(s)^{-1}\dot{V}(s)V(s)^{-1} \\
		&= \bigl(\dot{\Phi}_3 (s,t) + \dot{\Phi}_4 (s,t) Q(t)\bigr)V^{-1}(s) - Q(s)\bigl(\dot{\Phi}_1 (s,t) + \dot{\Phi}_2 (s,t) Q(t)\bigr)V^{-1}(s).
	\end{align*}
	Evaluating at $s = t$ and noting that $V(t) = I$, we find
	\begin{align*}
		\dot{Q}(t) &= \bigl(\dot{\Phi}_3(t,t) + \dot{\Phi}_4(t,t)Q(t)\bigr)V^{-1}(t) - Q(t)\bigl(\dot{\Phi}_1(t,t) + \dot{\Phi}_2(t,t)Q(t)\bigr)V^{-1}(t) \\
		&= -H_{xx}(\gamma(t),p(t)) - H_{xp}(\gamma(t),p(t))Q(t) - Q(t)H_{px}(\gamma(t),p(t)) - Q(t)H_{pp}(\gamma(t),p(t))Q(t).
	\end{align*}
	Thus, $Q(t)$ satisfies the matrix Riccati equation \eqref{eq:Ricc}. This completes the proof of (1).
	
We now turn to prove (2). If $t_0=0$, the assertion follows from part (1). We therefore assume $0<t_0<T$. If $\gamma(t_0)\in\operatorname{supp}_{C^{1,1}}(u(t_0,\cdot))$, according to \cite[Theorem 2.3]{Albano2014_1} and \cite[Proposition 3.2]{Cannarsa_Cheng_Hong2025}, there exists $\delta>0$ such that $t_0+\delta<T$ and
\begin{equation}\label{eq:alex pf 11}
	(t,\gamma(t))\in\operatorname{supp}_{C^{1,1}}(u)=\big([0,+\infty)\times\T^n\big)\setminus\overline{\cut(u,H)},\qquad \forall t\in[0,t_0+\delta].
\end{equation}
We first claim that there exists a neighborhood $W$ of $\gamma(0)$ such that
\begin{equation}\label{eq:backward tau}
	\tau_{\phi,H}(y)\geqslant t_0,\qquad \forall y\in W.
\end{equation}
Suppose, to the contrary, that there exists a sequence $y_i\to\gamma(0)$ with $\tau_{\phi,H}(y_i)<t_0$ and $\tau_{\phi,H}(y_i)\to\bar t\leqslant t_0$. If $\tau_{\phi,H}(y_i)=0$ along a subsequence, then $(0,y_i)\in\cut(u,H)$, and letting $i\to\infty$ we get $(0,\gamma(0))\in\overline{\cut(u,H)}$, which contradicts \eqref{eq:alex pf 11}. Thus, up to a subsequence, $0<\tau_{\phi,H}(y_i)<t_0$. By Lemma \ref{lem:uniq cali}, there exists a $(u,H)$-calibrated curve $\gamma_i:[0,\tau_{\phi,H}(y_i)]\to\T^n$ with $\gamma_i(0)=y_i$. If $(\tau_{\phi,H}(y_i),\gamma_i(\tau_{\phi,H}(y_i)))$ were regular, a $(u,H)$-calibrated continuation for a positive time interval would extend $\gamma_i$ beyond $\tau_{\phi,H}(y_i)$, contradicting the maximality in Definition \ref{defn:tau_phi}. Hence
\begin{equation}\label{eq:alex pf 12}
	(\tau_{\phi,H}(y_i),\gamma_i(\tau_{\phi,H}(y_i)))\in\cut(u,H),\qquad \forall i\in\N.
\end{equation}
For all $i\in\N$, since $y_i$ is a regular point of $\phi$, Lemma \ref{lem:uniq cali} yields the representation
\begin{equation}\label{eq:backward flow}
	\gamma_i(s)=\pi\circ\Phi^s_H(y_i,D\phi(y_i)),\ \forall s\in[0,\tau_{\phi,H}(y_i)],\quad \gamma(s)=\pi\circ\Phi^s_H(\gamma(0),D\phi(\gamma(0))),\ \forall s\in[0,t_0],
\end{equation}
where $\pi:T^*\T^n\to\T^n$ is the canonical projection and $\Phi^s_H$ is the Hamiltonian flow for $H$. The upper semicontinuity of $D^+\phi$ gives $D\phi(y_i)\to D\phi(\gamma(0))$. Together with \eqref{eq:backward flow} and $\tau_{\phi,H}(y_i)\to\bar t$, we obtain $(\tau_{\phi,H}(y_i),\gamma_i(\tau_{\phi,H}(y_i)))\to(\bar t,\gamma(\bar t))$. Combing this with \eqref{eq:alex pf 12}, it follows that $(\bar t,\gamma(\bar t))\in\overline{\cut(u,H)}$, which contradicts \eqref{eq:alex pf 11} above because $\bar t\leqslant t_0<t_0+\delta$. Therefore \eqref{eq:backward tau} holds.

For each $y\in W$ we select the unique $(u,H)$-calibrated curve $\gamma_y:[0,t_0]\to\T^n$ with $\gamma_y(0)=y$, which exists and is unique by \eqref{eq:backward tau} and Lemma \ref{lem:uniq cali}. Applying the separation estimate in Theorem \ref{distance estimate} to the calibrated curve $\gamma_y$ on $[0,t_0]$ and $\gamma$ on $[0,T]$ with $t=t_0$, we get
\begin{equation}\label{eq:backward sep}
	\frac{(T-t_0)^{C_2}}{e^{C_1 T}}\cdot d(y,\gamma(0))\leqslant d(\gamma_y(t_0),\gamma(t_0))\leqslant\frac{e^{C_1 T}}{(T-t_0)^{C_2}}\cdot d(y,\gamma(0)).
\end{equation}
In particular, the endpoints $\gamma_y(t_0)$ remain close to $\gamma(t_0)$ uniformly as $y\to\gamma(0)$.

Let $p_y$ be the dual arc of $\gamma_y$. By Proposition \ref{pro:cali} (2) and (4), $p_y(0)=D\phi(y)$, the pair $(\gamma_y,p_y)$ satisfies the Hamiltonian system on $[0,t_0]$, and $p_y(t_0)\in D^+_xu(t_0,\gamma_y(t_0))$. Set $\bar\gamma(s)=\gamma_y(s)-\gamma(s)$ and $\bar p(s)=p_y(s)-p(s)$ for $s\in[0,t_0]$. Since $\gamma(t_0)\in\operatorname{Alex}(u(t_0,\cdot))$, Proposition \ref{pro:Alex2} yields
\begin{equation}\label{eq:backward alex}
	\bar p(t_0)=Q(t_0)\bar\gamma(t_0)+o(|\bar\gamma(t_0)|).
\end{equation}
By the Lipschitz dependence of the Hamiltonian flow and \eqref{eq:backward alex}, $\max\{|\bar\gamma(s)|,|\bar p(s)|\}\leqslant C|\bar\gamma(t_0)|$ for all $s\in[0,t_0]$ when $|\bar\gamma(t_0)|$ is sufficiently small. Thus, as in the proof of (1), the pair $(\bar\gamma,\bar p)$ satisfies the linearized system
\begin{equation}\label{eq:alex pf 13}
	\begin{pmatrix}
		\dot{\bar{\gamma}}(s)\\
		\dot{\bar{p}}(s)
	\end{pmatrix}
	=
	\begin{pmatrix}
		H_{px}(\gamma(s),p(s)) & H_{pp}(\gamma(s),p(s))\\
		-H_{xx}(\gamma(s),p(s)) & -H_{xp}(\gamma(s),p(s))
	\end{pmatrix}
	\cdot
	\begin{pmatrix}
		\bar{\gamma}(s)\\
		\bar{p}(s)
	\end{pmatrix}+o(|\bar\gamma(t_0)|),\quad s\in[0,t_0].
\end{equation}
Since all the curves under consideration lie in a compact region of the phase space $T^*\T^n$, the Taylor remainder of the Hamiltonian vector field is uniform, and \eqref{eq:backward sep} makes the remainders in \eqref{eq:backward alex} and \eqref{eq:alex pf 13} uniform as $y\to\gamma(0)$. Let $\Phi(0,t_0)$ be the inverse of the state-transition matrix $\Phi(t_0,0)$ of the linearized system \eqref{eq:alex pf 13} on $[0,t_0]$, and let $\Phi_k(0,t_0)$, $k=1,2,3,4$, be its $n\times n$ blocks. Solving backward and using \eqref{eq:backward alex}, we obtain
\begin{equation}\label{eq:backward transport}
	\begin{cases}
		\bar\gamma(0)=\bigl(\Phi_1(0,t_0)+\Phi_2(0,t_0)Q(t_0)\bigr)\bar\gamma(t_0)+o(|\bar\gamma(t_0)|),\\
		\bar p(0)=\bigl(\Phi_3(0,t_0)+\Phi_4(0,t_0)Q(t_0)\bigr)\bar\gamma(t_0)+o(|\bar\gamma(t_0)|).
	\end{cases}
\end{equation}

We claim that the matrix $N=\Phi_1(0,t_0)+\Phi_2(0,t_0)Q(t_0)$ is invertible. 
Suppose, to the contrary, that $N$ is degenerate, and choose a unit vector $\theta\in\R^n$ with $N\theta=0$. By Proposition \ref{pro:cali} (1), for any $\tau>0$ there exists a $(u,H)$-calibrated curve $\tilde\gamma_\tau:[0,t_0]\to\T^n$ with $\tilde\gamma_\tau(t_0)=\gamma(t_0)+\tau\theta$ in local coordinate. Applying Theorem \ref{distance estimate} directly to the calibrated curve $\tilde\gamma_\tau$ and to $\gamma$ with $t=t_0$ gives $d(\tilde\gamma_\tau(0),\gamma(0))\leqslant\frac{e^{C_1T}}{(T-t_0)^{C_2}}\tau$, so $\tilde\gamma_\tau(0)\in W$ for all $\tau$ small. By Lemma \ref{lem:uniq cali}, $\tilde\gamma_\tau$ coincides with the curve $\gamma_y$ for $y=\tilde\gamma_\tau(0)$, and \eqref{eq:backward transport} applies, giving $\tilde\gamma_\tau(0)-\gamma(0)=N(\tau\theta)+o(\tau)=o(\tau)$. On the other hand, \eqref{eq:backward sep} yields
\[
d(\tilde\gamma_\tau(0),\gamma(0))\geqslant\frac{(T-t_0)^{C_2}}{e^{C_1T}}\cdot d(\tilde\gamma_\tau(t_0),\gamma(t_0))=\frac{(T-t_0)^{C_2}}{e^{C_1T}}\tau,
\]
a contradiction. Hence $N$ is invertible.

Therefore, by \eqref{eq:backward transport}, \eqref{eq:backward sep} and the invertibility of $N$, setting
\[
Q(0)=\bigl(\Phi_3(0,t_0)+\Phi_4(0,t_0)Q(t_0)\bigr)N^{-1},
\]
we obtain
\[
\bar p(0)=Q(0)\bar\gamma(0)+o(|\bar\gamma(0)|),
\]
that is, since $\bar\gamma(0)=y-\gamma(0)$ and $\bar p(0)=D\phi(y)-D\phi(\gamma(0))$,
\[
D\phi(y)-D\phi(\gamma(0))=Q(0)(y-\gamma(0))+o(|y-\gamma(0)|),\qquad \forall y\in W.
\]
Since $\phi$ is differentiable on $W$, Proposition \ref{pro:Alex2} implies $\gamma(0)\in\operatorname{Alex}(u(0,\cdot))$. Thus, assertion (2) follows from (1).
\end{proof}

\begin{Rem}
	The main idea for the proof of Theorem \ref{the:propagate alex} is, for any $y$ near $\gamma(t)$, finding a $(u,H)$-calibrated curve $\tilde{\gamma}$ connecting $y$ at time $t$ to a point near $\gamma(t_0)$ at time $t_0$. In this way, we can obtain the information at time $t$ from the information at time $t_0$, by a method of differences. The existence of $\tilde{\gamma}$ is trivial for $t>t_0$ (Proposition \ref{pro:cali} (1)). However, for $t<t_0$, such $\tilde{\gamma}$ may not exist in general. This is the only difference between the two cases, and is why we make the $C^{1,1}$ support assumption in (2).
\end{Rem}

Theorem \ref{the:propagate alex} and Theorem \ref{thm:close_cut} immediately imply the following corollary.
\begin{Cor}
	Let $H$ be a Tonelli Hamiltonian on $\T^n$ and let $\phi$ be a viscosity solution of \eqref{eq:HJs}.
	Then the following statements are true:
	\begin{enumerate}
		\item If $\gamma: \mathbb{S}^1 \to \T^n$ is a periodic $(\phi,H)$-calibrated curve and $\gamma(t_0) \in \operatorname{Alex}(\phi)$ for some $t_0 \in \mathbb{S}^1$, then $\gamma(t) \in \operatorname{Alex}(\phi)$ for all $t \in \mathbb{S}^1$, and $Q(t)=D^2\phi(\gamma(t))$, $t \in \mathbb{S}^1$ is of class $C^1$ and satisfies the matrix Riccati equation \eqref{eq:Ricc}.
		\item  Suppose  $\cut(\phi)$ is closed. If $\gamma: (-\infty,0] \to \T^n$ is a $(\phi,H)$-calibrated curve and $\gamma(t_0) \in \operatorname{Alex}(\phi)$ for some $t_0\in(-\infty,0)$, then $\gamma(t) \in \operatorname{Alex}(\phi)$ for all $t \in (-\infty,0)$, and $Q(t)=D^2\phi(\gamma(t))$, $t \in (-\infty,0)$ is of class $C^1$ and satisfies the matrix Riccati equation \eqref{eq:Ricc}.
	\end{enumerate}
\end{Cor}

The following theorem is a direct consequence of Proposition \ref{pro:Alex2} and Theorem \ref{thm:1} (1).
\begin{The}\label{thm:up of alex}
	Let $H$ be a Tonelli Hamiltonian on $\T^n$ and $\phi\in\SCL(\T^n)$. Then there exist $C_1>0$ and $C_2>0$ (depending on the local $C^2$ modulus and the Tonelli modulus of $H$ and the Lipschitz and semiconcavity constants of $\phi$) such that for any $x\in\operatorname{Alex}(\phi)$
	\begin{align*}
		\|D^2\phi(x)\|\leqslant\max\Big\{\ \frac{C_1}{\tau_{\phi,H}(x)},C_2\ \Big\}.
	\end{align*}
\end{The}

\begin{Rem}
	Theorem \ref{thm:up of alex} gives an upper bound of the second-order differential of a semiconcave function. For Theorem \ref{the:propagate alex}, Proposition \ref{pro:uni lip semi} implies $u(t,\cdot)$ is uniformly Lipschitz and uniformly semiconcave with linear modulus for all $t\geqslant0$. Therefore, if $D_x^2 u(t,\gamma(t))$ blows up as $t\to t_0$, the rate is no more than $\frac{1}{t_0-t}$, which is exactly the blow up rate of the solution of Riccati equation.
\end{Rem}

\bibliographystyle{plain}
\bibliography{mybib}

@article{Azagra2023,
	author = {Azagra, Daniel and Cappello, Anthony and Hajłasz, Piotr},
	doi = {10.1090/bproc/190},
	fjournal = {Proceedings of the American Mathematical Society, Series B},
	issn = {2330-1511},
	journal = {Proc. Amer. Math. Soc. Ser. B},
	pages = {382--397},
	title = {A geometric approach to second-order differentiability of convex functions},
	volume = {10},
	year = {2023},
	}

@article{Cannarsa_Cheng_Hong2025,
	author = {Cannarsa, Piermarco and Cheng, Wei and Hong, Jiahui},
	doi = {10.1016/j.nonrwa.2024.104282},
	fjournal = {Nonlinear Analysis. Real World Applications. An International Multidisciplinary Journal},
	issn = {1468-1218},
	journal = {Nonlinear Anal. Real World Appl.},
	mrclass = {35F21 (49L25)},
	mrnumber = {4836625},
	pages = {Paper No. 104282},
	title = {Topological and control theoretic properties of {H}amilton--{J}acobi equations via {L}ax-{O}leinik commutators},
	doi = {10.1016/j.nonrwa.2024.104282},
	url = {https://doi.org/10.1016/j.nonrwa.2024.104282},
	volume = {84},
	year = {2025}}

@article{Figalli_Rifford_Villani2011,
	author = {Figalli, A. and Rifford, L. and Villani, C.},
	doi = {10.1016/j.difgeo.2011.02.002},
	fjournal = {Differential Geometry and its Applications},
	issn = {0926-2245},
	journal = {Differential Geom. Appl.},
	mrclass = {53C20 (53C22)},
	mrnumber = {2784296},
	mrreviewer = {Luca Granieri},
	number = {2},
	pages = {154--159},
	title = {Tangent cut loci on surfaces},
	url = {https://mathscinet.ams.org/mathscinet-getitem?mr=2784296},
	volume = {29},
	year = {2011}}

@unpublished{CCHW2024,
	author = {Cannarsa, Piermarco and Cheng, Wei and Hong, Jiahui and Wang, Kaizhi},
	note = {preprint, arXiv:2409.00961},
	title = {Variational construction of singular characteristics and propagation of singularities},
	year = {2024}}

@article{Arnaud2011,
	author = {Arnaud, M.-C.},
	doi = {10.1088/0951-7715/24/1/003},
	fjournal = {Nonlinearity},
	issn = {0951-7715},
	journal = {Nonlinearity},
	mrclass = {37J50 (35F21 37J40 70H20)},
	mrnumber = {2739831},
	mrreviewer = {Karl Friedrich Siburg},
	number = {1},
	pages = {71--78},
	title = {Pseudographs and the {L}ax-{O}leinik semi-group: a geometric and dynamical interpretation},
	url = {https://mathscinet.ams.org/mathscinet-getitem?mr=2739831},
	volume = {24},
	year = {2011}}

@article{Cannarsa_Cheng_Fathi2021,
	author = {Cannarsa, Piermarco and Cheng, Wei and Fathi, Albert},
	doi = {10.1007/s10240-021-00125-5},
	fjournal = {Publications Math\'{e}matiques. Institut de Hautes \'{E}tudes Scientifiques},
	issn = {0073-8301},
	journal = {Publ. Math. Inst. Hautes \'{E}tudes Sci.},
	mrclass = {Prelim},
	mrnumber = {4292741},
	number = {1},
	pages = {327--366},
	title = {Singularities of solutions of time dependent {H}amilton-{J}acobi equations. {A}pplications to {R}iemannian geometry},
	url = {https://mathscinet.ams.org/mathscinet-getitem?mr=4292741},
	volume = {133},
	year = {2021}}

@article{Cannarsa_Cheng2021a,
	author = {Cannarsa, Piermarco and Cheng, Wei},
	doi = {10.1007/s00032-021-00330-1},
	fjournal = {Milan Journal of Mathematics},
	issn = {1424-9286},
	journal = {Milan J. Math.},
	mrclass = {35F21 (37J51 49L25)},
	mrnumber = {4277365},
	number = {1},
	pages = {187--215},
	title = {Singularities of {S}olutions of {H}amilton--{J}acobi {E}quations},
	url = {https://mathscinet.ams.org/mathscinet-getitem?mr=4277365},
	volume = {89},
	year = {2021}}

@article{Bogaevsky2002,
	author = {Bogaevsky, Ilya Aleksandrovich},
	doi = {10.1016/S0167-2789(02)00652-8},
	fjournal = "{Physica D. Nonlinear Phenomena}",
	issn = {0167-2789},
	journal = {Phys. D},
	mrclass = {58K50 (35Q35 35Q53 76B99)},
	mrnumber = {1945478},
	mrreviewer = {Vladimir V. Tchernov},
	number = {1-2},
	pages = {1--28},
	title = {Perestroikas of shocks and singularities of minimum functions},
	url = {https://mathscinet.ams.org/mathscinet-getitem?mr=1945478},
	volume = {173},
	year = {2002}}

@article{Colombo_Marigonda2006,
	author = {Colombo, Giovanni and Marigonda, Antonio},
	doi = {10.1007/s00526-005-0352-7},
	fjournal = "{Calculus of Variations and Partial Differential Equations}",
	issn = {0944-2669},
	journal = {Calc. Var. Partial Differential Equations},
	mrclass = {49J50 (26B25 35A15 49J52 58C20)},
	mrnumber = {2183853},
	mrreviewer = {Alberto Zaffaroni},
	number = {1},
	pages = {1--31},
	title = {Differentiability properties for a class of non-convex functions},
	url = {https://mathscinet.ams.org/mathscinet-getitem?mr=2183853},
	volume = {25},
	year = {2006}}

@article{Ancona_Cannarsa_Nguyen2016_1,
	author = {Ancona, Fabio and Cannarsa, Piermarco and Nguyen, Khai T.},
	doi = {10.1007/s00205-015-0907-5},
	fjournal = "{Archive for Rational Mechanics and Analysis}",
	issn = {0003-9527},
	journal = {Arch. Ration. Mech. Anal.},
	mrclass = {35F21 (35F25 49L25)},
	mrnumber = {3437863},
	mrreviewer = {Piotr Rybka},
	number = {2},
	pages = {793--828},
	title = {Quantitative compactness estimates for {H}amilton-{J}acobi equations},
	url = {https://mathscinet.ams.org/mathscinet-getitem?mr=3437863},
	volume = {219},
	year = {2016}}

@article{Ancona_Cannarsa_Nguyen2016_2,
	author = {Ancona, Fabio and Cannarsa, Piermarco and Nguyen, Khai T.},
	fjournal = "{Bulletin of the Institute of Mathematics. Academia Sinica. New Series}",
	issn = {2304-7909},
	journal = {Bull. Inst. Math. Acad. Sin. (N.S.)},
	mrclass = {35F21 (35D40 35F25 47B06 49L20)},
	mrnumber = {3497746},
	mrreviewer = {Qihuai Liu},
	number = {1},
	pages = {63--113},
	title = {Compactness estimates for {H}amilton-{J}acobi equations depending on space},
	url = {https://mathscinet.ams.org/mathscinet-getitem?mr=3497746},
	volume = {11},
	year = {2016}}

@book{Krylov_book1987,
	author = {Krylov, N. V.},
	doi = {10.1007/978-94-010-9557-0},
	isbn = {90-277-2289-7},
	mrclass = {35-02 (35J60 35K55)},
	mrnumber = {901759},
	note = {Translated from the Russian by P. L. Buzytsky [P. L. Buzytski\u{\i}]},
	pages = {xiv+462},
	publisher = {D. Reidel Publishing Co., Dordrecht},
	series = {Mathematics and its Applications (Soviet Series)},
	title = {Nonlinear elliptic and parabolic equations of the second order},
	url = {https://mathscinet.ams.org/mathscinet-getitem?mr=901759},
	volume = {7},
	year = {1987}}

@article{Douglis1961,
	author = {Douglis, Avron},
	doi = {10.1002/cpa.3160140307},
	fjournal = "{Communications on Pure and Applied Mathematics}",
	issn = {0010-3640},
	journal = {Comm. Pure Appl. Math.},
	mrclass = {35.20},
	mrnumber = {139848},
	mrreviewer = {O. A. Ole\u{\i}nik},
	pages = {267--284},
	title = {The continuous dependence of generalized solutions of non-linear partial differential equations upon initial data},
	url = {https://mathscinet.ams.org/mathscinet-getitem?mr=139848},
	volume = {14},
	year = {1961}}

@incollection{Petrunin2007,
	author = {Petrunin, Anton},
	booktitle = {Surveys in differential geometry. {V}ol. {XI}},
	doi = {10.4310/SDG.2006.v11.n1.a6},
	mrclass = {53C21 (53C20 53C45 58E05)},
	mrnumber = {2408266},
	mrreviewer = {Jeremy Wong},
	pages = {137--201},
	publisher = {Int. Press, Somerville, MA},
	series = {Surv. Differ. Geom.},
	title = {Semiconcave functions in {A}lexandrov's geometry},
	url = {https://mathscinet.ams.org/mathscinet-getitem?mr=2408266},
	volume = {11},
	year = {2007}}

@incollection{Rockafellar1982,
	author = {Rockafellar, R. Tyrrell},
	booktitle = {Progress in nondifferentiable optimization},
	mrclass = {90C48 (49A52 58C20 90C30)},
	mrnumber = {704977},
	pages = {125--143},
	publisher = {Internat. Inst. Appl. Systems Anal., Laxenburg},
	series = {IIASA Collaborative Proc. Ser. CP-82},
	title = {Favorable classes of {L}ipschitz-continuous functions in subgradient optimization},
	url = {https://mathscinet.ams.org/mathscinet-getitem?mr=704977},
	volume = {8},
	year = {1982}}

@article{Rifford2000,
	author = {Rifford, Ludovic},
	doi = {10.1137/S0363012999356039},
	fjournal = "{SIAM Journal on Control and Optimization}",
	issn = {0363-0129},
	journal = {SIAM J. Control Optim.},
	mrclass = {93D05 (34D20 49J52 49L25 93D20)},
	mrnumber = {1814266},
	mrreviewer = {Zvi Artstein},
	number = {4},
	pages = {1043--1064},
	title = {Existence of {L}ipschitz and semiconcave control-{L}yapunov functions},
	url = {https://mathscinet.ams.org/mathscinet-getitem?mr=1814266},
	volume = {39},
	year = {2000}}

@article{Rifford2002,
	author = {Rifford, Ludovic},
	doi = {10.1137/S0363012900375342},
	fjournal = "{SIAM Journal on Control and Optimization}",
	issn = {0363-0129},
	journal = {SIAM J. Control Optim.},
	mrclass = {93D30 (34D20 34H05 49L25 93C15 93D15)},
	mrnumber = {1939865},
	mrreviewer = {Zvi Artstein},
	number = {3},
	pages = {659--681},
	title = {Semiconcave control-{L}yapunov functions and stabilizing feedbacks},
	url = {https://mathscinet.ams.org/mathscinet-getitem?mr=1939865},
	volume = {41},
	year = {2002}}

@article{Fleming_McEneaney2000,
	author = {Fleming, Wendell H. and McEneaney, William M.},
	doi = {10.1137/S0363012998332433},
	fjournal = "{SIAM Journal on Control and Optimization}",
	issn = {0363-0129},
	journal = {SIAM J. Control Optim.},
	mrclass = {93E11 (49L25)},
	mrnumber = {1741434},
	mrreviewer = {Jose Luis Menaldi},
	number = {3},
	pages = {683--710},
	title = {A max-plus-based algorithm for a {H}amilton-{J}acobi-{B}ellman equation of nonlinear filtering},
	url = {https://mathscinet.ams.org/mathscinet-getitem?mr=1741434},
	volume = {38},
	year = {2000}}

@article{Cannarsa_Frankowska1991,
	author = {Cannarsa, Piermarco and Frankowska, Halina},
	doi = {10.1137/0329068},
	fjournal = "{SIAM Journal on Control and Optimization}",
	issn = {0363-0129},
	journal = {SIAM J. Control Optim.},
	mrclass = {49K40 (49L20 49N35)},
	mrnumber = {1132185},
	mrreviewer = {Vera Zeidan},
	number = {6},
	pages = {1322--1347},
	title = {Some characterizations of optimal trajectories in control theory},
	url = {https://mathscinet.ams.org/mathscinet-getitem?mr=1132185},
	volume = {29},
	year = {1991}}

@article{Hrustalev1978,
	author = {Hrustal\"{e}v, M. M.},
	fjournal = "{Doklady Akademii Nauk SSSR}",
	issn = {0002-3264},
	journal = {Dokl. Akad. Nauk SSSR},
	mrclass = {49B10},
	mrnumber = {510254},
	mrreviewer = {I. V. Romanovski\u{\i}},
	number = {5},
	pages = {1023--1026},
	title = {Necessary and sufficient conditions for optimality in the form of the {B}ellman equation},
	url = {https://mathscinet.ams.org/mathscinet-getitem?mr=510254},
	volume = {242},
	year = {1978}}

@book{Berger_book2003,
	author = {Berger, Marcel},
	doi = {10.1007/978-3-642-18245-7},
	isbn = {3-540-65317-1},
	mrclass = {53-02 (53C20 53C21 53C55 58J20)},
	mrnumber = {2002701},
	mrreviewer = {J\"{u}rgen Eichhorn},
	pages = {xxiv+824},
	publisher = {Springer-Verlag, Berlin},
	title = {A panoramic view of {R}iemannian geometry},
	url = {https://mathscinet.ams.org/mathscinet-getitem?mr=2002701},
	year = {2003}}

@book{Ambrosio_GigliNicola_Savare_book2008,
	author = {Ambrosio, Luigi and Gigli, Nicola and Savar\'{e}, Giuseppe},
	edition = {Second},
	isbn = {978-3-7643-8721-1},
	mrclass = {49-02 (28A33 35K55 35K90 49Q20 60B05)},
	mrnumber = {2401600},
	mrreviewer = {Pietro Celada},
	pages = {x+334},
	publisher = {Birkh\"{a}user Verlag, Basel},
	series = {Lectures in Mathematics ETH Z\"{u}rich},
	title = {Gradient flows in metric spaces and in the space of probability measures},
	url = {https://mathscinet.ams.org/mathscinet-getitem?mr=2401600},
	year = {2008}}

@article{Rifford2003,
	author = {Rifford, Ludovic},
	doi = {10.1512/iumj.2003.52.2287},
	fjournal = {Indiana University Mathematics Journal},
	issn = {0022-2518},
	journal = {Indiana Univ. Math. J.},
	mrclass = {49L25 (93D15)},
	mrnumber = {2010731},
	mrreviewer = {Guy Jumarie},
	number = {5},
	pages = {1373--1395},
	title = {Singularities of some viscosity supersolutions and the stabilization problem in the plane},
	url = {https://mathscinet.ams.org/mathscinet-getitem?mr=2010731},
	volume = {52},
	year = {2003}}

@article{Itoh_Tanaka2001,
	author = {Itoh, Jin-ichi and Tanaka, Minoru},
	doi = {10.1090/S0002-9947-00-02564-2},
	fjournal = {Transactions of the American Mathematical Society},
	issn = {0002-9947},
	journal = {Trans. Amer. Math. Soc.},
	mrclass = {53C20 (26A16 28A78 53C22)},
	mrnumber = {1695025},
	mrreviewer = {James J. Hebda},
	number = {1},
	pages = {21--40},
	title = {The {L}ipschitz continuity of the distance function to the cut locus},
	url = {https://mathscinet.ams.org/mathscinet-getitem?mr=1695025},
	volume = {353},
	year = {2001}}

@book{Villani_book2009,
	author = {Villani, C\'{e}dric},
	doi = {10.1007/978-3-540-71050-9},
	isbn = {978-3-540-71049-3},
	mrclass = {49-02 (28A75 37J50 49Q20 53C23 58E30)},
	mrnumber = {2459454},
	mrreviewer = {Dario Cordero-Erausquin},
	pages = {xxii+973},
	publisher = {Springer-Verlag, Berlin},
	series = {Grundlehren der Mathematischen Wissenschaften},
	title = {Optimal transport: old and new},
	url = {https://mathscinet.ams.org/mathscinet-getitem?mr=2459454},
	volume = {338},
	year = {2009}}

@book{Fleming_Soner_book2006,
	author = {Fleming, Wendell H. and Soner, Halil Mete},
	edition = {Second},
	isbn = {978-0387-260457; 0-387-26045-5},
	mrclass = {93-02 (49L20 49L25 60J60 90C39 91A23 91B28 93E20)},
	mrnumber = {2179357},
	pages = {xviii+429},
	publisher = {Springer, New York},
	series = {Stochastic Modelling and Applied Probability},
	title = {Controlled {M}arkov processes and viscosity solutions},
	url = {https://mathscinet.ams.org/mathscinet-getitem?mr=2179357},
	volume = {25},
	year = {2006}}

@article{Li_Nirenberg2005,
	author = {Li, Y. and Nirenberg, L.},
	fjournal = {Communications on Pure and Applied Mathematics},
	issn = {0010-3640},
	journal = {Comm. Pure Appl. Math.},
	mrclass = {35F20 (53C60)},
	mrnumber = {2094267},
	mrreviewer = {Jie Yang},
	number = {1},
	pages = {85--146},
	title = {The distance function to the boundary, {F}insler geometry, and the singular set of viscosity solutions of some {H}amilton-{J}acobi equations},
	url = {https://doi.org/10.1002/cpa.20051},
	volume = {58},
	year = {2005}}

@article{ACNS2013,
	author = {Albano, Paolo and Cannarsa, Piermarco and Nguyen, Khai Tien and Sinestrari, Carlo},
	fjournal = {Mathematische Annalen},
	issn = {0025-5831},
	journal = {Math. Ann.},
	mrclass = {35D40 (26B25 35A20 35A21 49J52)},
	mrnumber = {3038120},
	mrreviewer = {Luigi Rodino},
	number = {1},
	pages = {23--43},
	title = {Singular gradient flow of the distance function and homotopy equivalence},
	url = {https://doi.org/10.1007/s00208-012-0835-8},
	volume = {356},
	year = {2013}}

@article{Albano2014_1,
	author = {Albano, Paolo},
	fjournal = "{Journal of Mathematical Analysis and Applications}",
	issn = {0022-247X},
	journal = {J. Math. Anal. Appl.},
	mrclass = {35F21 (35D40 35F25)},
	mrnumber = {3128423},
	mrreviewer = {Gawtum Namah},
	number = {2},
	pages = {684--687},
	title = {Propagation of singularities for solutions of {H}amilton-{J}acobi equations},
	doi = {10.1016/j.jmaa.2013.10.015},
	url = {https://doi.org/10.1016/j.jmaa.2013.10.015},
	volume = {411},
	year = {2014}}

@article{Albano_Cannarsa2002,
	author = {Albano, Paolo and Cannarsa, Piermarco},
	fjournal = "{Archive for Rational Mechanics and Analysis}",
	issn = {0003-9527},
	journal = {Arch. Ration. Mech. Anal.},
	mrclass = {35A20 (35F20)},
	mrnumber = {1892229},
	mrreviewer = {Luigi Rodino},
	number = {1},
	pages = {1--23},
	title = {Propagation of singularities for solutions of nonlinear first order partial differential equations},
	doi = {10.1007/s002050100176},
	url = {https://doi.org/10.1007/s002050100176},
	volume = {162},
	year = {2002}}

@book{Bardi_Capuzzo-Dolcetta1997,
	author = {Bardi, Martino and Capuzzo-Dolcetta, Italo},
	isbn = {0-8176-3640-4},
	mrclass = {49-02 (49J15 49K15 49L25)},
	mrnumber = {1484411},
	mrreviewer = {Vladimir Veliov},
	note = {With appendices by Maurizio Falcone and Pierpaolo Soravia},
	pages = {xviii+570},
	publisher = {Birkh{\"a}user Boston, Inc., Boston, MA},
	series = {Systems \& Control: Foundations \& Applications},
	title = {Optimal control and viscosity solutions of {H}amilton-{J}acobi-{B}ellman equations},
	doi = {10.1007/978-0-8176-4755-1},
	url = {https://doi.org/10.1007/978-0-8176-4755-1},
	year = {1997}}

@article{Bernard2007,
	author = {Bernard, Patrick},
	fjournal = {Annales Scientifiques de l'{\'E}cole Normale Sup{\'e}rieure. Quatri{\`e}me S{\'e}rie},
	issn = {0012-9593},
	journal = {Ann. Sci. {\'E}cole Norm. Sup. (4)},
	mrclass = {37J50 (35F20 49L25 70H20)},
	mrnumber = {2493387},
	mrreviewer = {Daniel Massart},
	number = {3},
	pages = {445--452},
	title = {Existence of {$C^{1,1}$} critical sub-solutions of the {H}amilton-{J}acobi equation on compact manifolds},
	url = {https://doi.org/10.1016/j.ansens.2007.01.004},
	volume = {40},
	year = {2007}}

@article{Bernard2012,
	author = {Bernard, Patrick},
	fjournal = {Proceedings of the Royal Society of Edinburgh. Section A. Mathematics},
	issn = {0308-2105},
	journal = {Proc. Roy. Soc. Edinburgh Sect. A},
	mrclass = {70H08 (37J40 70H20)},
	mrnumber = {3002592},
	mrreviewer = {Jos{\'e} Claudio Vidal Diaz},
	number = {6},
	pages = {1131--1177},
	title = {The {L}ax-{O}leinik semi-group: a {H}amiltonian point of view},
	doi = {10.1017/S0308210511000059},
	url = {https://doi.org/10.1017/S0308210511000059},
	volume = {142},
	year = {2012}}

@article{Cannarsa_Cheng3,
	author = {Cannarsa, Piermarco and Cheng, Wei},
	fjournal = "{Calculus of Variations and Partial Differential Equations}",
	issn = {0944-2669},
	journal = {Calc. Var. Partial Differential Equations},
	mrclass = {35F21 (37J50 49L25)},
	mrnumber = {3687884},
	number = {5},
	pages = {Art. 125, 31},
	title = {Generalized characteristics and {L}ax-{O}leinik operators: global theory},
	doi = {10.1007/s00526-017-1219-4},
	url = {https://doi.org/10.1007/s00526-017-1219-4},
	volume = {56},
	year = {2017}}

@article{Cannarsa_Mazzola_Sinestrari2015,
	author = {Cannarsa, Piermarco and Mazzola, Marco and Sinestrari, Carlo},
	fjournal = {Discrete and Continuous Dynamical Systems. Series A},
	issn = {1078-0947},
	journal = {Discrete Contin. Dyn. Syst.},
	mrclass = {35F21 (35A21 35F25 49J52)},
	mrnumber = {3392624},
	number = {9},
	pages = {4225--4239},
	title = {Global propagation of singularities for time dependent {H}amilton-{J}acobi equations},
	url = {https://doi.org/10.3934/dcds.2015.35.4225},
	volume = {35},
	year = {2015}}

@book{Cannarsa_Sinestrari_book,
	author = {Cannarsa, Piermarco and Sinestrari, Carlo},
	isbn = {0-8176-4084-3},
	mrclass = {49-02 (35F20 49K20 49L20)},
	mrnumber = {2041617},
	mrreviewer = {Pierre Cardaliaguet},
	pages = {xiv+304},
	publisher = {Birkh{\"a}user Boston, Inc., Boston, MA},
	series = {Progress in Nonlinear Differential Equations and their Applications},
	title = {Semiconcave functions, {H}amilton-{J}acobi equations, and optimal control},
	volume = {58},
	year = {2004}}

@article{Cannarsa_Yu2009,
	author = {Cannarsa, Piermarco and Yu, Yifeng},
	fjournal = {Journal of the European Mathematical Society (JEMS)},
	issn = {1435-9855},
	journal = {J. Eur. Math. Soc. (JEMS)},
	mrclass = {49K20 (26B25 35F21 49L20 49L25)},
	mrnumber = {2538498},
	mrreviewer = {Pietro Celada},
	number = {5},
	pages = {999--1024},
	title = {Singular dynamics for semiconcave functions},
	url = {https://doi.org/10.4171/JEMS/173},
	volume = {11},
	year = {2009}}

@article{Crandall_Evans_Lions1984,
	author = {Crandall, Michael G. and Evans, Lawrence C. and Lions, Pierre-Louis},
	fjournal = "{Transactions of the American Mathematical Society}",
	issn = {0002-9947},
	journal = {Trans. Amer. Math. Soc.},
	mrclass = {35F20 (35L60)},
	mrnumber = {732102},
	number = {2},
	pages = {487--502},
	title = {Some properties of viscosity solutions of {H}amilton-{J}acobi equations},
	doi = {10.2307/1999247},
	url = {https://doi.org/10.2307/1999247},
	volume = {282},
	year = {1984}}

@article{Crandall_Lions1983,
	author = {Crandall, Michael G. and Lions, Pierre-Louis},
	fjournal = "{Transactions of the American Mathematical Society}",
	issn = {0002-9947},
	journal = {Trans. Amer. Math. Soc.},
	mrclass = {35F20},
	mrnumber = {690039},
	mrreviewer = {Moshe Marcus},
	number = {1},
	pages = {1--42},
	title = {Viscosity solutions of {H}amilton-{J}acobi equations},
	doi = {10.2307/1999343},
	url = {https://doi.org/10.2307/1999343},
	volume = {277},
	year = {1983}}

@article{Dafermos1977,
	author = {Dafermos, Constantine M.},
	fjournal = "{Indiana University Mathematics Journal}",
	issn = {0022-2518},
	journal = {Indiana Univ. Math. J.},
	mrclass = {35L65},
	mrnumber = {0457947},
	mrreviewer = {J. Smoller},
	number = {6},
	pages = {1097--1119},
	title = {Generalized characteristics and the structure of solutions of hyperbolic conservation laws},
	doi = {10.1512/iumj.1977.26.26088},
	url = {https://doi.org/10.1512/iumj.1977.26.26088},
	volume = {26},
	year = {1977}}

@article{Fathi_Siconolfi2005,
	author = {Fathi, Albert and Siconolfi, Antonio},
	fjournal = "{Calculus of Variations and Partial Differential Equations}",
	issn = {0944-2669},
	journal = {Calc. Var. Partial Differential Equations},
	mrclass = {35F20 (37J50 49K10)},
	mrnumber = {2106767},
	mrreviewer = {Andrzej Swiech},
	number = {2},
	pages = {185--228},
	title = {P{DE} aspects of {A}ubry-{M}ather theory for quasiconvex {H}amiltonians},
	doi = {10.1007/s00526-004-0271-z},
	url = {https://doi.org/10.1007/s00526-004-0271-z},
	volume = {22},
	year = {2005}}

@article{Khanin_Sobolevski2016,
	author = {Khanin, Konstantin and Sobolevski, Andrei},
	fjournal = "{Archive for Rational Mechanics and Analysis}",
	issn = {0003-9527},
	journal = {Arch. Ration. Mech. Anal.},
	mrclass = {35F21 (49L25 70H20 76B03)},
	mrnumber = {3437865},
	mrreviewer = {Qihuai Liu},
	number = {2},
	pages = {861--885},
	title = {On dynamics of {L}agrangian trajectories for {H}amilton-{J}acobi equations},
	doi = {10.1007/s00205-015-0910-x},
	url = {https://doi.org/10.1007/s00205-015-0910-x},
	volume = {219},
	year = {2016}}

@article{Kruzkov1975,
	author = {Kru\v{z}kov, S. N.},
	journal = {Mat. Sb. (N.S.)},
	mrclass = {35L60},
	mrnumber = {0404870},
	mrreviewer = {V. Ju. Ljapidevskii},
	number = {3(11)},
	pages = {450--493, 496},
	title = {Generalized solutions of {H}amilton-{J}acobi equations of eikonal type. {I}. {S}tatement of the problems; existence, uniqueness and stability theorems; certain properties of the solutions},
	volume = {98(140)},
	year = {1975}}

@article{Bogaevski2006,
	author = {Bogaevsky, Ilya Aleksandrovich},
	doi = {10.4213/sm1502},
	fjournal = {Rossi\u\i skaya Akademiya Nauk. Matematicheski\u\i \ Sbornik},
	issn = {0368-8666},
	journal = {Mat. Sb.},
	mrclass = {34A12 (28B20 34A60 47H05 49J52 49K15)},
	mrnumber = {2437079},
	mrreviewer = {Dariusz Zagrodny},
	number = {12},
	pages = {11--42},
	title = {Discontinuous gradient differential equations, and trajectories in the calculus of variations},
	doi = {10.4213/sm1502},
	url = {https://doi.org/10.1070/SM2006v197n12ABEH003820},
	volume = {197},
	year = {2006}}

@article{Rifford2008,
	author = {Rifford, Ludovic},
	fjournal = {Communications in Partial Differential Equations},
	issn = {0360-5302},
	journal = {Comm. Partial Differential Equations},
	mrclass = {49L25 (35B65 35F20)},
	mrnumber = {2398240},
	number = {1-3},
	pages = {517--559},
	title = {On viscosity solutions of certain {H}amilton-{J}acobi equations: regularity results and generalized {S}ard's theorems},
	url = {https://doi.org/10.1080/03605300701382522},
	volume = {33},
	year = {2008}}

@article{Stromberg2013,
	author = {Str{\"o}mberg, Thomas},
	fjournal = {Nonlinear Analysis. Theory, Methods \& Applications. An International Multidisciplinary Journal},
	issn = {0362-546X},
	journal = {Nonlinear Anal.},
	mrclass = {35F21 (35A20 35D40)},
	mrnumber = {3040351},
	mrreviewer = {Fabiana Leoni},
	pages = {93--109},
	title = {Propagation of singularities along broken characteristics},
	doi = {10.1016/j.na.2013.02.024},
	url = {https://doi.org/10.1016/j.na.2013.02.024},
	volume = {85},
	year = {2013}}

@article{Yu2006,
	author = {Yu, Yifeng},
	fjournal = {Annali della Scuola Normale Superiore di Pisa. Classe di Scienze. Serie V},
	issn = {0391-173X},
	journal = {Ann. Sc. Norm. Super. Pisa Cl. Sci. (5)},
	mrclass = {35F20 (35A21)},
	mrnumber = {2297718},
	number = {4},
	pages = {439--444},
	title = {A simple proof of the propagation of singularities for solutions of {H}amilton-{J}acobi equations},
	volume = {5},
	year = {2006}}

@article{Miura_Tanaka2024,
	author = {Miura, Tatsuya and Tanaka, Minoru},
	title = {Delta-convex structure of the singular set of distance functions},
	journal = {Comm. Pure Appl. Math.},
	volume = {77},
	number = {9},
	pages = {3631--3669},
	year = {2024},
	doi = {10.1002/cpa.22195},
}

@article{Bianchini_Tonon2012,
	author = {Bianchini, Stefano and Tonon, Daniela},
	title = {{SBV} regularity for {H}amilton--{J}acobi equations with {H}amiltonian depending on $(t,x)$},
	journal = {SIAM J. Math. Anal.},
	volume = {44},
	number = {3},
	pages = {2179--2203},
	year = {2012},
	doi = {10.1137/110827272},
}

@book{Klingenberg_book1982,
	author = {Klingenberg, Wilhelm},
	title = {Riemannian Geometry},
	series = {De Gruyter Studies in Mathematics},
	volume = {1},
	publisher = {Walter de Gruyter \& Co., Berlin},
	year = {1982},
	note = {Second edition, 1995}
}

@book{Sakai_book1996,
	author = {Sakai, Takashi},
	title = {Riemannian Geometry},
	series = {Translations of Mathematical Monographs},
	volume = {149},
	publisher = {American Mathematical Society, Providence, RI},
	year = {1996},
	note = {Translated from the 1992 Japanese original}
}

@article{Cheng_Wei2026,
	author = {Cheng, Wei and Wei, Wenxue},
	title = {A geometric approach to the {M}ather quotient problem},
	journal = {Sci. China Math.},
	fjournal = {Science China Mathematics},
	issn = {1674-7283},
	volume = {69},
	number = {4},
	pages = {943--960},
	year = {2026}
}

@unpublished{Cheng_Hong2026,
	author = {Cheng, Wei and Hong, Jiahui},
	title = {Recent progress in generalized Hamiltonian gradient flow: singularities},
	note = {preprint, arXiv:2605.04658},
	year = {2026}
}

@article{Albano_Cannarsa_Cheng_Mendico2026,
	author = {Albano, Paolo and Cannarsa, Piermarco and Cheng, Wei and Mendico, Cristian},
	title = {Long-time behavior of generalized gradient flows of solutions to {H}amilton--{J}acobi equations},
	journal = {Commun. Contemp. Math.},
	fjournal = {Communications in Contemporary Mathematics},
	year = {2026},
	note = {online}
}

@article{Albano_Cannarsa_Sinestrari2020,
	author = {Albano, Paolo and Cannarsa, Piermarco and Sinestrari, Carlo},
	title = {Generation of singularities from the initial datum for {Hamilton--Jacobi} equations},
	journal = {J. Differential Equations},
	volume = {268},
	number = {4},
	pages = {1412--1426},
	doi = {10.1016/j.jde.2019.08.051},
	year = {2020}}
\end{document}